\documentclass[a4paper,12pt]{amsart} 
\usepackage{soul}
\usepackage{amsmath, amssymb, amsfonts, amsthm, amssymb}
\usepackage{euscript, mathrsfs, stackrel}
\usepackage{enumerate}
\usepackage{stmaryrd}
\usepackage[utf8]{inputenc}   
\usepackage[T1]{fontenc}      
\usepackage{color}
\usepackage{graphicx}
\usepackage{hyperref}
\usepackage{wrapfig}
\usepackage{cleveref}

\theoremstyle{plain}
\numberwithin{equation}{section}
\newtheorem{theorem}{Theorem}[section]

\newtheorem{proposition}[theorem]{Proposition}
\newtheorem{lemma}[theorem]{Lemma}

\newtheorem{remark}[theorem]{Remark}
\usepackage[left=2.5cm,right=2.5cm,vmargin=2.5cm]{geometry}
\usepackage{stmaryrd, mathtools}

\DeclarePairedDelimiterX\intff[2]{[}{]}{#1,#2}
\DeclarePairedDelimiterX\intfo[2]{[}{)}{#1,#2}
\DeclarePairedDelimiterX\intof[2]{(}{]}{#1,#2}
\DeclarePairedDelimiterX\intoo[2]{(}{)}{#1,#2}
\DeclarePairedDelimiter{\pars}{(}{)}
\DeclarePairedDelimiter{\bracks}{[}{]}
\DeclarePairedDelimiter{\absolute}{|}{|}

\DeclarePairedDelimiterX{\setof}[2]{\lbrace}{\rbrace}{#1\,{:}\,#2}
\DeclarePairedDelimiterX{\bracksof}[2]{[}{]}{#1\,\delimsize\vert\,#2}
\DeclarePairedDelimiterX{\parsof}[2]{(}{)}{#1\,\delimsize\vert\,#2}
\DeclarePairedDelimiterXPP\lnorm[2]{}\lVert\rVert{_{#1}}{#2}

\usepackage{mathrsfs}
\def\n{\mathbb N}
\def\z{\mathbb Z}
\def\r{\mathbb R}
\def\e{\mathbb E}
\def\p{\mathbb P}
\def\P{\mathbf P}
\def\E{\mathbf E}
\def\RR{\mathfrak R}
\def\BScap{\mbox{\rm BScap}}
\def\Bcap{\mbox{\rm Bcap}}

\def\ball{{\rm B}}
\def\T{{\mathcal T}}
\def\X{{\mathcal X}}
\def\Y{{\mathcal Y}}
\def\dd{{\rm d}}
\def\brwrange{{\mathfrak R}}
\renewcommand{\Cap}{\mathrm{Cap}}
\newcommand{\rd}{\mathrm{d}}
\newcommand{\anc}{\mathrm{anc}}
\newcommand{\bt}{\mathbf{t}}

\newcommand{\cU}{\mathcal{U}}
\newcommand{\QQ}{\Upsilon}
\newcommand{\cnbd}[2]{{{#1}^{#2}}}

\def\brwrange{{\mathscr R}}
\def\brwm{M_\theta}

\newcommand{\ttree}[1]{\mathcal T_{#1}}
\newcommand{\tbrw}[2]{V_{#1}}
\newcommand{\httree}[1]{\widehat{\mathcal T}_{#1}}
\newcommand{\htbrw}[2]{\widehat{V}_{#1}}
\newcommand{\ottree}[1]{\protect\overleftarrow{\mathcal T_{#1}}}

\newcommand{\rwp}{P^{(\xi)}}
\newcommand{\rwe}{E^{(\xi)}}

\def\wcvn{\, {\stackrel[n\to\infty]{\text (law)}{\longrightarrow}}\, }
\def\diam{{\rm diam}}
\def\N{\mathcal N}

\title{Branching capacity of  a random walk in $\z^5$}
\author{Tianyi Bai}
\address{Tianyi Bai, Chinese Academy of Sciences,   China.}
\email{tianyi.bai73@amss.ac.cn}

\author{Jean-Fran\c{c}ois Delmas}
\address{Jean-Fran\c{c}ois Delmas, CERMICS,
Ecole des Ponts,  France}
\email{delmas@cermics.enpc.fr}

\author{Yueyun Hu}
\address{Yueyun Hu, 
LAGA, Universit\'e Paris XIII,  93430 Villetaneuse,
France}
\email{yueyun@math.univ-paris13.fr}

\begin{document}

\begin{abstract}  We are  interested in  the branching  capacity of  the
  range     of     a     random     walk     in     $\z^d$.     Schapira
  \cite{schapira2023branching} has recently obtained precise asymptotics
  in the  case $d\ge 6$ and  has demonstrated a transition  at dimension
  $d=6$.  We  study the  case  $d=5$  and  prove that  the  renormalized
  branching capacity converges in law  to the Brownian snake capacity of
  the range of a  Brownian motion. The main step in  the proof relies on
  studying the intersection probability between  the range of a critical
  Branching  random  walk  and  that  of a  random  walk,  which  is  of
  independent interest.
\end{abstract}

\subjclass[2010]{60F05, 60J45, 60J80}

\keywords{Branching capacity, range, random walk, critical branching random walk.}

\maketitle

\section{Introduction}

  Let $(\xi_n)_{n\ge 0}$ be a centered, aperiodic and irreducible  random walk in $\z^d$ with finite second moment, whose law is denoted by $\rwp$. Denote by $\xi[0,n]:= \{\xi_0, ..., \xi_n\}$ the range of $\xi$ up to time $n$. 
Studying the range $\xi[0, n]$ is a classical problem in probability theory. 

Concerning the asymptotics of  the size of the range $\#\xi[0,n]$, it is well-known  (Dvoretzky and Erd\H{o}s \cite{DE51})   that there is a transition at dimension $d=2$: $\#\xi[0,n]$ is of order $n$ when $d\ge 3$ and of order $\frac{n}{\log n}$ when $d=2$;  we refer to Jain and Orey \cite{JO73},  Jain and Pruitt \cite{JP71}, and Le Gall \cite{LG86a, LG86b} for deep studies on the size of the range.

The  capacity of  the range  depends on  its geometry  and has  recently
attracted significant  interest. The discrete Newtonian  capacity can be
defined as  follows.  Let $d\ge  3$ and $K  \subset \z^d$ be  a nonempty
finite  set.  Let  ${\mathcal S}$  be the  range of  a simple  symmetric
random  walk  on $\z^d$  starting  from  $x$  whose  law is  denoted  by
${P}_x$. The discrete Newtonian capacity of $K$, ${\rm Cap}(K)$, can be
defined   for    $d\ge   3$   (up   to    a   multiplicative   constant)
as  \begin{equation}  {\rm  Cap}(K):= \lim_{x  \in  \z^d,\,  x\to\infty}
  |x|^{d-2} {P}_x({\mathcal S}\cap K\neq \emptyset), \label{def-cap}
 \end{equation}
 
 \noindent  where $|x|$  denotes the  usual  Euclidean norm  of $x$  and
 $x\to\infty$ means  that $|x|\to \infty$.  Following the works  of Jain
 and Orey \cite{JO68}, Asselah, Schapira and Sousi \cite{ASS18, ASS19},
 and Chang \cite{chang17}, it is  known that ${\rm Cap}(\xi[0,n])$ is of
 order $n$ when $d\ge 5$, of order $\frac{n}{\log n}$ when $d=4$, and of
 order  $\sqrt{n}$ when  $d=3$. This  implies a  transition at  dimension
 $d=4$.  See   Asselah,  Schapira,   and  Sousi   \cite{AS20},  Schapira
 \cite{schapira20}  and the  references  therein for  the central  limit
 theorems, and Dembo and Okada \cite{DO22+} for the laws of the iterated
 logarithm. Additionally, Asselah and Schapira \cite{AS20+} investigated
 the link  between the capacity and  the folding phenomenon of  a random
 walk, and Hutchcroft and Sousi \cite{HS23} explored the capacity of the
 loop-erased random walk.

 Recently,   Zhu  \cite{zhu2016critical}   introduced  the   concept  of
 branching capacity. The basic idea is  to replace the range of a simple
 random walk  ${\mathcal S}$  in \eqref{def-cap} by  that of  a critical
 branching  random walk.  Specifically,  let $\ttree  c$  be a  critical
 Galton-Watson   tree   with   offspring   distribution   $\mu$,   where
 $\mu=(\mu(i))_{i\ge 0}$ is a probability distribution on $\n$ such that
 $\sum_{i=0}^\infty  i   \mu(i)=1$.  To  avoid  triviality,   we  assume
 $\mu(1)< 1$.  In other words, $\ttree  c$ is a finite  random tree that
 starts  with one  particle  $\varnothing$, called  the  root, and  each
 particle independently produces a  random number of offspring according
 to $\mu$. The critical branching random walk on $\mathbb{Z}^d$, denoted
 by $V_c$, is  a random walk indexed by the  tree $\ttree c$ constructed
 as   follows.   Let  $\theta$   be   a   probability  distribution   on
 $\mathbb{Z}^d$,   representing   the   common   distribution   of   the
 displacements of $V_c$.  For each edge $e$ of $\ttree  c$, we assign an
 independent random  variable $X_e$  with distribution $\theta$.  We set
 $V_c(\varnothing):=x\in            \z^d$,            and            for
 $u\in \ttree  c \backslash \{\varnothing\}$, $V_c(u)=x+  \sum_{e} X_e$,
 where the sum is taken over all  edges $e$ belonging to the simple path
 in $\ttree c$ connecting $u$ to $\varnothing$. The range of $\tbrw c x$
 is denoted by \begin{equation} \brwrange_c:=\{V_c(u), u\in \ttree c\}\,
   \subset \, \z^d. \label{def-Rc}\end{equation}

 \noindent Denote  by $\mathbf  P_x$ the  law of $\tbrw  c x$  and write
 $\P=\P_0$.  Almost surely  the random tree $\ttree c$ is  finite, so is
 the range $\brwrange_c$.

 For $x\in \r^d$, let  $\ball(x,r):= \{y \in \r^d: |y-x| <r\}$ be the open ball centered at $x$ with radius $r> 0$.
We assume that $d\ge 5$, and for some $q>4$,  \begin{align}
& \text{$\mu$ has mean $1$ and variance $\sigma^2\in (0, \infty)$},
\label{hyp-tree} 
\\
\label{hyp-brw} 
& \text{$\theta$ is symmetric, irreducible with covariance matrix $\brwm$
                   and}\\
&   \nonumber         \text{there exists a finite constant $c$ such that for all $r>0$:}  \quad \theta\big(\ball(0, r)^c\big) \le c\, r^{-q}.
\end{align} 
The last condition in~\eqref{hyp-brw}  is in particular satisfied when $\theta$ has  a  finite $q$-th moment.

Denote by $P_x$ the law of a $\theta$-walk $(S_n)_{n\ge 0}$ on $\z^d$ started from $x$, meaning that the random walk $S$ has the step distribution $\theta$.  The Green function of $(S_n)$ is given by $g(x,y):=g(x-y)$ for any $x, y \in \z^d$, and 
\begin{equation}\label{c_g}
g(x):=\sum_{n=0}^\infty P_0(S_n=x)\sim 
c_{g}\, |x|_\theta^{2-d}\quad\text{as}\quad
 x\to \infty, \end{equation}
with 
\[
|x|_\theta:=  (x^T\brwm^{-1} x)^{1/2}  \quad \mbox{and} \quad c_{g}:= 
\frac{\Gamma(\frac {d-2} 2)}{2\pi^{d/2} \sqrt{\det\brwm}},
\] where the equivalence in \eqref{c_g} is given by  Uchiyama \cite[Theorem 2]{MR1625467}.
Let $K\subset \z^d$ be a nonempty finite set. By Zhu \cite{zhu2016critical}, when \eqref{hyp-brw} holds with $q=d$,  the following limit exists and is called the branching capacity of $K$:
\begin{equation}
  \label{def-bcap}
  \Bcap(K):= \lim_{x\in \z^d, \, x \to\infty} \frac{\P_x(\brwrange_c \cap K\neq \emptyset)}{g(x)}.
\end{equation}
Note that $\Bcap(\cdot)$ viewed as a set function is non-decreasing, invariant under translations, and strictly positive. We extend $\Bcap$ into a Choquet capacity on $\r^d$ by letting $\Bcap(A):= \Bcap(A\cap \z^d)$ for any $A\subset \r^d$.

The  branching  capacity,   studied  in  a  series  of   papers  by  Zhu
\cite{zhu2016critical,  zhu-2,  zhu2021critical},   has  also  been  the
subject  of  some  more  recent   works.   Asselah,  Shapira  and  Sousi
\cite{asselah2023local} have  shown that $\Bcap(K)$ can  be compared, up
to two positive  constants, with the discrete Riesz  capacity with index
$d-4$,  and have  revealed a  deep relationship  between the  local time
spent in a ball by the  branching random walk and the branching capacity
of this ball. Moreover, in another work, Asselah, Okada, Shapira and
Sousi  \cite{AOSS} have
demonstrated the comparability  of $\Bcap(K)$ with the  average limit of
the  size of  the Minkowski  sum of  $K$ and  two independent  copies of
$(S_n)$,  as  well as  with  the  hitting  probability  of $K$  by  this
Minkowski sum.  See also \cite{bdh-23+} for further references.

In \cite{bdh-23+}  we also proved  the vague convergence of  the scaling
limit of  $\Bcap$ towards  the capacity related  to the  Brownian snake,
denoted  by $\BScap$.  More  precisely, consider  the excursion  measure
$\mathcal  N_x$   given  by  the   distribution  of  a   Brownian  snake
$(W_t)_{t\ge 0}$ started  from $x\in \r^d$ whose lifetime process  is an It\^{o}
excursion.  Denote by $\mathfrak R$ the range of the Brownian snake, see
Le Gall  \cite[Chapter IV]{LeGall1999} for the  precise definitions.  It
was shown  in \cite{bdh-23+} that when  $d\ge 5$, for any  bounded Borel
set $A \subset \r^d$, the following limit exists and is finite:
\begin{equation}     \BScap(A):=\lim_{x\to\infty} |x|^{d-2}  \N_x(\RR\cap A\neq \emptyset).  \label{def-bscap}
 \end{equation}

\noindent We call $\BScap(A)$ the Brownian snake capacity of $A$.

The vague convergence in \cite[Theorem 1.4]{bdh-23+} says that for $d\ge 5$ and for any compact set $K\subset \r^d$ such that \begin{equation}  \BScap(\mathring K)= \BScap(K), \label{K-regular}  \end{equation} where $\mathring K$ denotes the interior of $K$, we have, under  \eqref{hyp-tree} and  \eqref{hyp-brw} with $q=d$, \begin{equation} \lim_{n\to\infty} \frac{\Bcap(\sqrt{n} K)}{n^{(d-4)/2}} = c_\theta \, \BScap(M_\theta^{-1/2} K),   \label{convergence-Kn}  \end{equation} 
 \noindent with 
 \begin{equation}  c_\theta:= \frac{2}{\sigma^2 c_g}= \frac {4\pi^{d/2}\, \sqrt{\det\brwm}}
{\sigma^2\, {\Gamma(\frac{d-2}{2})}}\cdot  \label{def-ctheta} \end{equation} 

We choose a renormalization of $n^{1/2}$ in \eqref{convergence-Kn} to ensure consistency with the choice of $K_n$ below.
The condition \eqref{K-regular} is in particular satisfied when $K$ is the closure of a bounded open set with Lipschitz boundary, see \cite[Proposition 1.3]{bdh-23+}.
It is a natural problem to investigate the branching capacity of (random) compact set  $K$ which does not satisfy \eqref{K-regular}. In this paper we consider the case $\sqrt{n} K_n:=  \xi[0,n]$ of the range of the random walk $(\xi_n)_{n\ge 0}$, which was recently studied by Schapira \cite{schapira2023branching}:  when both $(S_n)$ and $(\xi_n)$ are  simple symmetric random walks, the following asymptotics  hold: 
\begin{equation} \begin{aligned} 
&\lim_{n\to\infty} \frac1{n} \Bcap(\xi[0,n])  \text{ exists almost surely and is positive when $d\ge 7$};
\\
&\lim_{n\to\infty} \frac{\log n}{n} \Bcap(\xi[0,n]) = \frac{2 \pi^3}{27 \sigma^2} \text{ in probability when $d=6$};
\\
&\rwe[\Bcap(\xi[0,n])]  \text{ is of order $n^{1/2}$ when $d=5$.}
\end{aligned} 
 \end{equation}

\medskip

We aim to give a sharp result in dimension $d=5$.  Assume that  
\begin{equation}\label{hyp-xi}
\begin{aligned}
&(\xi_n)_{n\ge0}\, \text{ is aperiodic  irreducible and}\quad \rwe_0[{|\xi_1|}^3]<\infty,\\
&\rwe_0[\xi_1]=0 \quad\text{and}\quad
 \xi_1 \text{ has covariance matrix } M_\xi,
\end{aligned}
\end{equation}

\noindent where $\rwp_x$ means that the random walk
$(\xi_n)$  starts  from  $x\in  \z^d$. Let  $(\beta_t)_{t\ge  0}$  be  a
standard Brownian motion  in $\r^d$. For a real $d\times  d$ matrix $M$,
we define  $M \beta[0,1]:=\setof{M  \beta_t}{0\le t\le 1}$.  Notice that
$K_n= n^{-1/2} \xi[0,n]$  converges in law, for  the Hausdorff distance,
to $M_\xi^{1/2} \beta[0,1]$.

\begin{theorem} \label{t:L1} 
Let $d=5$. Assume \eqref{hyp-tree},  \eqref{hyp-brw} with $q=5$,  and \eqref{hyp-xi}. We have 
\[ 
\frac{\Bcap(\xi[0, n])}{\sqrt{n}} \, \wcvn\, \frac{8\pi^2 \sqrt{\det\brwm}}{\sigma^2}\,  \BScap(M_\theta^{-1/2}\,M_\xi^{1/2}\beta[0,1]).
\] Moreover, $\BScap(M_\theta^{-1/2}\, M_\xi^{1/2}\beta[0,1])$ is almost surely positive. 
\end{theorem}

\begin{remark}   
\begin{enumerate}
\item In Lemma \ref{lem:nonpoloar}, we will show that a.s.\  $\BScap(M_\theta^{-1/2}\, M_\xi^{1/2}\beta[0,1])>0$ when $d=5$, whereas it vanishes when $d\ge 6$.
\item  In the assumption \eqref{hyp-xi}, aperiodicity can be easily removed, and the third moment condition is required  in Lemma \ref{lem:reversed_intersection}.
\item In the case when both $(S_n)$ and $(\xi_n)$ are simple symmetric random walks, the limiting random variable in Theorem \ref{t:L1} becomes \[ 
\frac{\Bcap(\xi[0, n])}{\sqrt{n}} \, \wcvn\,  \frac{8\pi^2}{\sigma^2 d^{d/2}}\,  \BScap(\beta[0,1]).
\] 
\end{enumerate}
\end{remark}

\medskip
The key ingredient in the proof of Theorem \ref{t:L1} is the following estimate on the intersection probability between $(\xi_n)$ and the branching random walk $V_c$, which may be of independent interest. 

For       any      $A\subset       \r^d$       and      $r>0$,       let
$A^r:=\{x\in \r^d:  \dd(x, A)  \le r\}$  be the  closed $r$-neighborhood,
where $\dd(x, A):= \min_{y\in A} |x-y|$. For any $x\in \r^d$, we denote by
$\lfloor x\rfloor  $ the point  in $\mathbb{Z}^d$ whose  coordinates are
the      integer      parts      of     that      of      $x$,      thus
$|\lfloor x\rfloor - x| \le
\sqrt{d}$.  
We have the following estimate on the intersection probabilities between
$\brwrange_c$ and $\xi[0,n]$.

\begin{proposition}\label{p:intersection} Let $d=5$.  Assume \eqref{hyp-tree},  \eqref{hyp-brw} with $q>4$,  and \eqref{hyp-xi}. For any fixed $x\in\r^5\backslash\{0\}$ and any   $\eta\in(0,1)$,  we have 
\begin{equation}
\limsup_{\varepsilon\rightarrow 0+}\,\limsup_{n\rightarrow\infty}\,
n\, I(\varepsilon, n) = 0 , \label{eq:I->0}
\end{equation} where $x_n:=\lfloor{\sqrt n}\, x \rfloor \in \z^5$ and \begin{equation} I(\varepsilon, n):=\mathbf P_{x_n}\otimes \rwp_0 \pars*{ \brwrange_c \cap \cnbd{(\xi[0,n])}{{\varepsilon \sqrt n}} \not=\emptyset, \, \brwrange_c \cap \xi[0,n] =\emptyset,\,
\xi[0,n]\subset\ball\pars*{0,\eta{|x_n|}}}
.  \label{def-Iepsilon} \end{equation}
\end{proposition}

 \begin{remark}  
 
 (i) The condition $\xi[0,n]\subset\ball\pars*{0,\eta{|x_n|}}$ in $I(\varepsilon, n)$ guarantees that $\brwrange_c$ must have some growth to reach the $(\varepsilon\sqrt{n})$-neighborhood of $\xi[0,n]$. Without this condition, Proposition \ref{p:intersection} is no longer true,  see   Remark \ref{r:B-eta} for further details.
 
 (ii) We can deduce from \eqref{p:compareBcap} that at least for small $\eta$, 
 \begin{eqnarray*}   n^{-3/2}  \Bcap\big(\ball(0,  \varepsilon n^{1/2})\big) 
 &\lesssim  & 
 \mathbf P_{x_n}\otimes \rwp_0 \pars*{   \brwrange_c \cap \cnbd{(\xi[0,n])}{{\varepsilon \sqrt n}} \not=\emptyset, \, \xi[0,n]\subset\ball\pars*{0,\eta{|x_n|}}} 
  \\
  &\lesssim&   n^{-3/2}  \Bcap\big(\ball(0, (\eta |x|+\varepsilon) \sqrt{n})\big)
   .
 \end{eqnarray*}

\noindent The left and right terms of the above inequalities are, according to \eqref{convergence-Kn},   of order $\frac1{n}$.  This  explains the factor $n$ in \eqref{eq:I->0}.

 (iii)  
  When $d\ge 6$, \eqref{eq:I->0} also holds, under the conitions \eqref{hyp-tree},  \eqref{hyp-brw} with $q=d$,  and \eqref{hyp-xi}. In this case, the proof is elementary, see Lemma \ref{l:r1}.
\end{remark}

Let us say a few words on the proofs of Theorem  \ref{t:L1} and  Proposition  \ref{p:intersection}. Using Donsker's invariance principle and the forthcoming \eqref{p:compareBcap} and \eqref{Bcap->BScap},  we can  see that the scaling limit of ${\Bcap(\cnbd{\xi[0, n]}{\varepsilon n^{1/2}})}$  can be compared with   $\BScap(M_\theta^{-1/2}\, M_\xi^{1/2}\beta[0,1]).$  Therefore the proof of Theorem \ref{t:L1}  reduces to show that
$$
\Bcap(\cnbd{\xi[0, n]}{\varepsilon n^{1/2}})\approx
\Bcap({\xi[0, n]}),
$$
in the sense, see \eqref{def-bcap},  that for large $x$ 
 \[
\mathbf P_{x_n}\otimes \rwp_0\pars*{ \brwrange_c \cap \cnbd{(\xi[0,n])}{{\varepsilon n^{1/2}}} \not=\emptyset}
 \approx 
\mathbf P_{x_n}\otimes \rwp_0\pars*{ \brwrange_c \cap  \xi[0,n]  \not=\emptyset},
 \]
which is the content of Proposition \ref{p:intersection} (note that the condition $\xi[0,n]\subset\ball\pars*{0,\eta{|x_n|}}$ holds with high probability as $x$ is large).  To prove Proposition \ref{p:intersection}, we will switch the roles of $\xi$ and $\brwrange_c$ in the probability term of Proposition \ref{p:intersection} and study the probability of the form $\mathbf P_0 \otimes \rwp_{x_n}\big( \xi[0,n] \cap \cnbd{\brwrange_c}{{\varepsilon n^{1/2}}} \neq \emptyset,  {\xi[0,n]} \cap \brwrange_c = \emptyset \big)$. The latter probability can be estimated by a comparison to $\P_0 \otimes \rwp_{x_n}\big( \xi[0, \infty) \cap \brwrange_{n^2}=\emptyset\big),$ where $\brwrange_{n^2}$ denotes the range $\brwrange_c$ conditioned on the total population $\#\T_c=n^2$. This probability has already been studied in \cite{bai2022convergence}. To achieve such a comparison, we need to introduce some auxiliary trees and their associated branching random walks in Section \ref{sec:intersection_proba}, which will be the most technical part. We refer to Section  \ref{sec:3.2} for more detailed explanations on the proof of Proposition \ref{p:intersection}.

\medskip
The paper is organized as follows.
In Section \ref{sec:proof-BS-cap}, we recall some known results on the branching capacity and the Brownian snake capacity. In Section \ref{sec:convergence4}, we prove Theorem \ref{t:L1} by assuming Proposition \ref{p:intersection}. In Section \ref{sec:intersection_proba}, we introduce auxiliary branching random walks $V_+$, $V_-$, and $\widehat{V}_-$ in Section \ref{sec:tadj}, then recall several known results on $V_+$ and $\widehat{V}_-$ in Section \ref{sub:known}. After establishing some estimates on the increments of $V_+$ in Section \ref{sub:increment}, we present the key step in the proof of Proposition \ref{p:intersection} in Section \ref{sec:3.2}, involving a study of intersection probabilities between $V_\pm$ and the random walk $\xi$. Finally, we provide the proof of Proposition \ref{p:intersection} in Section \ref{sec:3.4}.

For notational convenience, 
we use the notations $C, C', C''$, eventually with some subscripts, to denote some positive constants whose values may vary from one paragraph to another.

\section{Brownian snake  capacity and Proof of  Theorem \ref{t:L1} by assuming Proposition \ref{p:intersection}}

\subsection{Branching capacity and Brownian snake  capacity.} 
\label{sec:proof-BS-cap}
We collect some facts on the Branching and the Brownian snake  capacities in this subsection. 
At first the following result (see \cite[Theorem 1.1]{bdh-23+})  provides a rate of convergence in \eqref{def-bcap} that will also be useful in our proof of Theorem \ref{t:L1}.  Let $d\ge 5$. Assume \eqref{hyp-tree},  \eqref{hyp-brw} with $q=d$. For any $\lambda>0$, there exist  some positive constants $\alpha$ and $C$ such that  for any $r\ge 1$,   $K \subset \ball(0,r) \cap \z^d$, $x\in \z^d$ with   $|x|\ge (1+\lambda)r$, we have
\begin{equation}  \Big|\Bcap(K)- \frac{ \P_x(\brwrange_c\cap K\neq \emptyset)}{g(x)} \Big| \le C \pars*{\frac{r}{|x|}}^\alpha\, \Bcap(K),
 \label{p:compareBcap} \end{equation}

\noindent where $\diam(K):=  \sup_{x, y\in K} |x-y|$  is the diameter of $K$. Moreover,  there are constants $c_1,c_2>0$, such that for any Borel set $A\subset \r^d$, 
\begin{equation}\label{eq:BScap_Cap_d-4}
c_1\Cap_{d-4}(A)\le \BScap(A)\le c_2\Cap_{d-4}(A), 
\end{equation}
 
 \noindent where for any $\gamma \in (0,d)$, 
\begin{equation}\label{def:capd=4}
    \Cap_\gamma(A):= \pars*{\inf_{\nu}\iint |x-y|^{-\gamma} \, \nu(\rd  x) \nu(\rd y)}^{-1}, 
\end{equation}
with  the infimum  taken  over all  the  probability measures  with
support in $A$.

 \begin{lemma}  Let $d\ge 5$. Assume \eqref{hyp-tree},  \eqref{hyp-brw} with $q=d$. For any compact  set $K\subset \r^d$, we have  \begin{equation} \limsup_{\varepsilon\to0} \limsup_{n\to\infty} \big| \frac{\Bcap(n\cnbd K \varepsilon)}{n^{d-4}}- c_\theta\BScap(M_\theta^{-1/2}K)\big|=0,  \label{Bcap->BScap}\end{equation} with  $ c_\theta $ given in \eqref{def-ctheta}.
\end{lemma}

\begin{proof}  First  $\BScap(\cdot)$ is a Choquet capacity relative to the set of compact sets on $\r^d$. In particular,    for every compact set $K\subset\mathbb R^d$, we have
 \begin{equation}  \lim_{\varepsilon\rightarrow 0+}\BScap(\cnbd K\varepsilon)
= \BScap(K) .  \label{eq:Kepsilon} \end{equation}

\noindent We check that the condition \cite[(1.7)]{bdh-23+} is satisfied for $K^\varepsilon$:   for any $y \in \partial K^\varepsilon$ and $n$ large enough, $K^\varepsilon\cap \ball(y, 2^{-n})$ contains a ball of radius $2^{-n-1}$ so  $\mbox{Cap}_{d-2}(K^\varepsilon\cap \ball(y, 2^{-n})) \ge \mbox{Cap}_{d-2}(\ball(0, 2^{-n-1}))= 2^{-(d-2)(n+1)} \mbox{Cap}_{d-2}(\ball(0, 1))$. Consequently for any $\varepsilon>0$, we have $\BScap(K^\varepsilon)=\BScap(K^{\varepsilon-})$, where $K^{\varepsilon-}=\{x \in \r^d: \dd(x, K)< \varepsilon\}$ the open $\varepsilon$-neighborhood of $K$.  Notice that these results also hold when $K$ is replaced by $\brwm^{-1/2} K$ and $K^\varepsilon$ by $\brwm^{-1/2} \, K^\varepsilon$.  This and \cite[Theorem 1.4]{bdh-23+} yield that 
 \begin{equation}\label{eq:Kepsilon2}
\lim_{n\rightarrow\infty}\frac{\Bcap(n\cnbd K \varepsilon)}{n^{d-4}}
= c_\theta\BScap(\brwm^{-1/2}\, \cnbd K\varepsilon).
\end{equation}

\noindent   Finally,  \eqref{Bcap->BScap} follows from \eqref{eq:Kepsilon2} and  \eqref{eq:Kepsilon}. 
 \end{proof}

\medskip
We end this subsection by the following result on the positivity of the Brownian snake capacity of the range of a Brownian motion. Recall that for a real  $d\times d$ matrix $M$, $M  \beta[0,1]=\setof{M  \beta_t}{0\le t\le 1}$.

\begin{lemma}
\label{lem:nonpoloar}  
Let $(\beta_t)_{t\ge 0}$ be a Brownian motion in $\r^d$ and $M$ be a
symmetric positive definite matrix $d\times d$. Then almost surely, 
 $$
\BScap(M   \beta[0,1]) \, 
\begin{cases}
>0, \qquad \mbox{ if $d=5$}, \\
=0, \qquad \mbox{ if $d\ge 6$}.
\end{cases}
$$
\end{lemma}

\begin{proof} By \cite[Theorem 1.1 and formula (9)]{PPS}, a.s., $\Cap_\gamma(\beta[0, 1])>0 $ if and only if $\gamma<\min(2, d)$. In particular, $\Cap_{d-4}(\beta[0, 1])$ is positive when $d=5$ and zero when $d\ge 6$. By \eqref{def:capd=4}, the same result holds when $\beta[0, 1]$ is replaced by $M\beta[0, 1]$,  because $\frac{|Mx-My|}{|x-y|}$ is uniformly bounded in $(0, \infty)$ for all $x\neq y$.  We conclude by using \eqref{eq:BScap_Cap_d-4}. \end{proof}

\subsection{Proof of  Theorem \ref{t:L1} by assuming Proposition \ref{p:intersection}} \label{sec:convergence4}

By  Donsker's   invariance  principle  and   Skorokhod's  representation
theorem, on a  common probability space $(\Omega, {\mathscr  F}, \p)$ we
may find a version of the random  walk $(\xi_n)$ starting from $0$ and a
standard Brownian  motion $(\beta_t)_{t\ge 0}$  in $\r^d$ such  that for
$\Theta:=\{M_\xi^{1/2}\beta_t: 0\le t \le  1\}$, almost surely for every
$\varepsilon>0$,  there  exists some  $n_0\ge  1$  such that  for  every
$n\ge n_0$,
\begin{equation}
n^{1/2}\cdot\Theta^{\varepsilon/2}\subset \cnbd{\xi[0,n]}{\varepsilon n^{1/2}}\subset n^{1/2}\cdot \cnbd\Theta{2\varepsilon}.\label{couplingBM0}
\end{equation}

\noindent Applying \eqref{Bcap->BScap} to $n^{1/2}\cdot \cnbd\Theta{\varepsilon/2}$ and $n^{1/2}\cdot \cnbd\Theta{2\varepsilon}$, we deduce that for  $d= 5$, $\p$-almost surely,    
\begin{align}\label{eq:a.s.}
\limsup_{\varepsilon\rightarrow 0+}\limsup_{n\rightarrow\infty}
\absolute*{
\frac{\Bcap(\cnbd{\xi[0,n]}{\varepsilon n^{1/2}})}{n^{{1/2}}}-c_\theta\BScap(M_\theta^{-1/2}\, \Theta)}=0.
\end{align}

Now we are ready to give the proof of Theorem \ref{t:L1}.

\begin{proof}[Proof of Theorem \ref{t:L1} by assuming Proposition \ref{p:intersection}]
Let $d=5$. By \eqref{eq:a.s.}, it suffices to show that  
\begin{align}\label{eq:e_cap_diff}
\limsup_{\delta\to0+} \limsup_{\varepsilon\rightarrow 0+}\limsup_{n\rightarrow\infty}\p\pars*{\Bcap(\cnbd{\xi[0,n]}{\varepsilon n^{1/2}})-\Bcap(\xi[0,n])>\delta n^{1/2}}=0.
\end{align}

Let $x\in \r^d\backslash\{0\}$ and  let $ x_n=\lfloor n^{1/2}x\rfloor$.
Let $\eta>0$ be small whose value will be determined later.
For $n$ large enough, on the event 
\[ E_n=E_n(x, \eta):= \Big\{\xi[0,n]\subset\ball(0,\eta|x_n|)  \Big\}, \]

\noindent by applying twice   \eqref{p:compareBcap} to obtain the following first and third inequalities, we have 
\begin{align*}
 \Bcap(\xi[0,n])   
&\ge  \frac{ \P_{x_n}\pars*{\brwrange_c \cap \xi[0,n]\neq \emptyset}}{ (1+C \eta^\alpha) g(x_n) }  
\\
&=  \frac{ \P_{x_n}\pars*{\brwrange_c \cap \cnbd{\xi[0,n]}{\varepsilon n^{1/2}}\neq \emptyset}}{(1+C \eta^\alpha)  g(x_n)  }  
 -  \frac{e_n(\varepsilon)}{1+C \eta^\alpha }
\\
&\ge  \frac{1-C   (\eta+2\varepsilon|x|^{-1})^\alpha}{1+C \eta^\alpha} \Bcap\pars*{\cnbd{\xi[0,n]}{\varepsilon n^{1/2}}}
 -   e_n(\varepsilon) ,
\end{align*}
 
\noindent where 
\begin{equation*}
e_n(\varepsilon):=\frac{1}{g(x_n) } \P_{x_n}\pars*{\brwrange_c \cap \xi[0,n]= \emptyset, \brwrange_c \cap \cnbd{\xi[0,n]}{\varepsilon n^{1/2}}\neq \emptyset}.
\end{equation*}

\noindent  For any  $\delta>0$, we  may find  and then  fix sufficiently
small $\eta=\eta(\delta)>0$ and $\varepsilon_0:=\varepsilon_0(\delta, |x|)>0$ such that
for    all   $\varepsilon\le    \varepsilon_0$,   we    have
$  \frac{1-C  (\eta+2\varepsilon|x|^{-1})^\alpha}{1+C  \eta^\alpha}  \ge
1-\delta^2$. It follows that
\begin{equation}\label{eq:error_e}
\begin{aligned}
&\p\pars*{\Bcap(\cnbd{\xi[0,n]}{\varepsilon n^{1/2}})-\Bcap(\xi[0,n])>\delta n^{1/2}}\\
\le&\p\pars*{
\delta^2\, \Bcap\pars*{\cnbd{\xi[0,n]}{\varepsilon n^{1/2}}}+e_n(\varepsilon)
>\delta n^{1/2},\, E_n}+\p(E_n^c)\\
\le&\p\pars*{
  n^{-1/2}\Bcap\pars*{\cnbd{\xi[0,n]}{\varepsilon n^{1/2}}}> \frac1{2\delta}}+
\p\pars*{e_n(\varepsilon)>\frac\delta2 n^{1/2},\, E_n}+\p(E_n^c)\\
=&:\eqref{eq:error_e}_1+\eqref{eq:error_e}_2+\eqref{eq:error_e}_3.
\end{aligned}
\end{equation}

\noindent By \eqref{eq:a.s.},
\begin{align*}
\limsup_{\varepsilon\rightarrow 0+}\limsup_{n\rightarrow\infty}\eqref{eq:error_e}_1
\le\p\pars*{\BScap(\Theta)>\frac{1}{4 c_\theta \delta}}
\end{align*}
is arbitrarily small as we take $\delta\rightarrow 0+$.

By Markov's inequality, $\eqref{eq:error_e}_2\le a_n    I(\varepsilon, n)$ where $I(\varepsilon, n)$ was defined in \eqref{def-Iepsilon} and  $a_n:= 2 \delta^{-1}    n^{-1/2} (g(x_n))^{-1} $ is of order of $n$ by  \eqref{c_g},  as $n\to\infty$ (recalling that $d=5$).  
By Proposition \ref{p:intersection},
\begin{align}\label{eq:bda_en2}
\limsup_{\varepsilon\rightarrow 0+}\limsup_{n\rightarrow\infty}\eqref{eq:error_e}_2\le C' \limsup_{\varepsilon\rightarrow 0+}\limsup_{n\rightarrow\infty} n \, I(\varepsilon, n)=0.
\end{align}

\noindent Finally for $\eqref{eq:error_e}_3$, we deduce from the standard random walk fluctuations that  
($\eta$ is fixed) $$
\limsup_{x\to \infty}\limsup_{n\rightarrow\infty} \p(E_n^c)=0.
$$
Combining the terms above and letting first $n\rightarrow\infty $, then
$\varepsilon\rightarrow 0$, $|x|\rightarrow \infty $ and lastly
$\delta\rightarrow 0$, we obtain \eqref{eq:e_cap_diff} and complete the proof of Theorem \ref{t:L1}.
\end{proof}

Recall  $I(\varepsilon, n)$,  defined  in \eqref{def-Iepsilon},  depends
also on $x$ and $\eta$. We  give an elementary proof of ~\eqref{eq:I->0}
when $d\geq 6$.

\begin{lemma}\label{l:r1}  Let $d\ge 6$.  Assume \eqref{hyp-tree},  \eqref{hyp-brw} with $q=d$. For any fixed $x\in\r^d\backslash\{0\}$ and any   $\eta\in(0,1)$,  we have $$ \limsup_{\varepsilon\to0} \limsup_{n\to\infty} n \, I(\varepsilon, n)=0.$$
  \end{lemma}   

\begin{proof} Let  $\Xi_n:= n^{(4-d)/2} \Bcap
  (\cnbd{(\xi[0,n])}{{\varepsilon \sqrt n}} ) \, 1_{\{
    \xi[0,n]\subset\ball(0, \eta |x_n|)\}}$.  By \eqref{p:compareBcap},
  there exists some positive constant $C=C(\eta, |x|)$ such that for all $n$,  $$
 n\,  I(\varepsilon, n) \le  C \,   \e (\Xi_n). $$ 

\noindent By \cite{zhu2016critical}, there exists some positive constant $C'$ such that for all $n$, $\Bcap(\ball(0, \eta |x_n|)) \le C'  |x_n|^{d-4}$.  We have $\Xi_n \le C'' $ for some positive constant $C''$ independent of $n$.  

Now we use  \eqref{couplingBM0} and \eqref{eq:Kepsilon2} to see that  $\p$-a.s.  
$$\limsup_{n\to\infty} \Xi_n \le 
\limsup_{n\to\infty}
n^{(4-d)/2} \Bcap( n^{1/2} \Theta^{2\varepsilon})
=c_\theta \BScap(M_\theta^{-1/2} \Theta^{2\varepsilon}).
$$

\noindent Applying Fatou's lemma to $C''-\Xi_n$, we get that $$\limsup_{n\to\infty} \e (\Xi_n) \le c_\theta \e (\BScap(M_\theta^{-1/2} \Theta^{2\varepsilon})).$$

\noindent Finally we remark that $\lim_{\varepsilon\to0} \e (\BScap(M_\theta^{-1/2} \Theta^{2\varepsilon}))= \e (\BScap(M_\theta^{-1/2} \Theta))=0$ for $d\ge 6$, by Lemma \ref{lem:nonpoloar}. This ends the proof. 
\end{proof}

\section{Intersection probabilities: Proof of Proposition \ref{p:intersection}}\label{sec:intersection_proba}

In this section,  we shall consider planar tree  using the lexicographic
order in Ulam-Harris setting.   We set $\cU^*=\cup_{n\in \n^*} (\n^*)^n$
and $\cU=\cU^* \cup \{\varnothing\}$,  where $\varnothing$ is called the
root.   For   $u=u_1  \cdots   u_n\in  \cU^*$,   the  set   of  ancestors
$\anc(u)\subset       \cU$      of       $u$      is       given      by
$\anc(u)=\{u_1  \cdots  u_k\,  \colon\,  k\in \{1,  \ldots,  n\}\}  \cup
\{\varnothing\}$.
We see a rooted planar tree  $\bt$ as a
subset    of    $\cU$    such    that:    $\varnothing\in    \bt$;    if
$u\in   \bt\cap \cU^*$   then 
$\anc(u) \subset  \bt$;  if $u\in  \bt$,  then there  exists
$k\in \n$  such that $ui\in \bt$  for all $i\in \{1,  \ldots, k\}$ (with
the convention that  $\varnothing i=i$). The tree $\bt$  is endowed with
the  usual   lexicographic  order~$\prec$ with  the   convention  that
 $\varnothing\prec u$ for
all $u\in \cU^*$.

We set $\cU'=\cU\cup \cU^*_-$, with   $\cU_-^*=\{-u\, \colon u\in
\cU^*\}$, where for the word $u=u_1\cdots u_n\in \cU^*$, we define the word  
$   -u:=(-u_1) \cdots (-u_n)$.
We extend the  lexicographic order $\prec$ on $\cU$ to $\cU'$ as
follows. For $u\in \cU^*_-$, we have $u\prec v$ if: either $v\in \cU$; or
$v\in\cU^*_-$ and $-u\in \anc(-v)$; or $v\in\cU^*_-$ and $-u\not\in
\anc(-v)$  and $-v\prec -u$ (in $\cU^*$). For example, we have:
\[
  (-5)\,\prec\, (-1) \, \prec\, (-1)(-1)(-1)\, \prec\, \varnothing
  \,\prec\, 1
  \, \prec\,   111 \,\prec\, 5.
\]

\subsection{On the discrete tree models}\label{sec:tadj} We will introduce the adjoint tree $\ttree{\text{adj}}$ and invariant tree $\ttree{\infty}$, as well as the associated branching random walks.

 The random planer  tree  $\ttree{\text{adj}}$ is  derived from the
 Galton-Watson tree $\ttree{c}$, with the only modification being made at the root. In the adjoint tree, the root has an offspring distribution $$\widetilde{\mu}(k)=\sum_{j=k+1}^\infty \mu(j), \qquad k\ge 0,$$ instead of $\mu$, while all other vertices retain the original offspring distribution $\mu$. 

 We then construct  the invariant tree $\ttree\infty$  as follows. Start
 with    an    infinite    spine   from    the    root    $\varnothing$:
 \[
   \X=\{\varnothing_0:=\varnothing,   \varnothing_1,   ...,   \varnothing_n,
   ...\}.
 \]
 To $\varnothing$,  we graft on the right a  planar random tree
 $\T^\text{d}_0$   distributed  as   the  critical   Galton-Watson  tree
 $\ttree{c}$     (identifying     the     root    of     $\T_c$     with
 $\varnothing=\varnothing_0$).   Then,  for  any $n\ge  1$,  the  vertex
 $\varnothing_n$ has $i$ children on the  left and $j$ on the right with
 probability   $\mu(i+j+1)$,  and   we  graft   independent  copies   of
 $\ttree{c}$ to  each child  (identifying the root  of the  grafted tree
 with the child).   As a remark, the two planar  subtrees grafted to the
 left,   $\T_n^\text{g}$,  and   to  the   right,  $\T_n^\text{d}$,   of
 $\varnothing_n$ (with  their root identified with  $\varnothing_n$) are
 dependent   random   trees  and,   for   $n\geq   1$,  distributed   as
 $\ttree{\text{adj}}$.  For  $n=0$, the  tree $\T_0^\text{g}$  is simply
 reduced to its root.

 An    element    $u\neq    \varnothing$     of    the    planar    tree
 $\T_n^\text{d}\subset  \cU$  is  coded  by  the  word  $  n  \,  u$  in
 $\T_\infty $,  and an  element $u\neq \varnothing$  of the  planar tree
 $\T_n^\text{g}\subset  \cU$  is  coded  by  the  word  $  (-n)(-u)$  in
 $\T_\infty $,  and $\varnothing_n$  is coded by  $(-n)$ for  $n\geq 1$.
 One can  see the  tree $\T_\infty $  as a subset  of $\cU'$.   Then, we
 order the vertices  of $\T_\infty $ by the  lexicographic order $\prec$
 on $\cU'$:
\[
 \dots\prec \ttree \infty{}(-2)\prec\ttree \infty{}(-1)\prec\ttree
 \infty{}(0)\prec\ttree \infty{}(1)\prec \cdots,
\]
with $\T_\infty  (0)=\varnothing_0=\varnothing$. Moreover, we  denote by
$\ttree+:=\T_\infty     (\n)=      \T_\infty     \cap      \cU$     (and
$\ttree-:=\T_\infty      (\z_-)=\T_\infty     \cap      \cU_-$     where
$\cU_-=\cU_-^*    \cup\{\varnothing\} $) the subgraph  of $\ttree\infty$
with  non-negative (and  non-positive)  labels.  Notice
 $\T_-$ is a tree and 
that the  spine
$\X$ is a subset
of  $\T_-$, but   $\ttree  +$  is
disconnected and thus no longer a tree, see Fig.~\ref{fig:tree_models}. 

The    sequence
  $\T_\infty(0),  \T_\infty(-1),  \T_\infty(-2),  \dots$ is  not  a  depth-first 
sequence for $\ttree-$. We shall  reorganize $\ttree -$ in depth-first
order by considering the planar rooted tree
$\httree-\subset \cU$ built as the union of the spine
$\X$ and the trees $(\T^\text{g}_n)_{n\geq 0}$,
where 
$\varnothing_{n+1}$ is  identified as an extra oldest child of the root of
$\T^\text{g}_n$ (which is still identified with $\varnothing_n$). 
We then  order  the vertices of $\httree-$ using the lexicographic
order $\prec$ on $\cU$:
\[
 \varnothing=\httree -(0)\prec \varnothing_1=\httree -(1)\prec\httree
 -(2)\prec \cdots. 
\]
See an illustration in Fig.~\ref{fig:tree_models}. 

\medskip

For   each  $\alpha\in\{c,\text{adj},\infty,   +,  -\}$,   we  construct
$\tbrw \alpha {}$, the branching  random walk indexed by $\ttree \alpha$
with displacement distribution  $\theta$, in the same way as  we did for
$V_c$ and $\T_c$.   For notational brevity, for any $i\in  \z$, we write
$V_\infty(i):=V_\infty(\T_\infty(i))$ the spatial  position of the $i$th
vertex of $\T_\infty$.
For $a,b\in\z$, we denote
\[
\tbrw \infty{}[a,b]=\{\tbrw \infty{}(i),a\le i\le b\},
\]
and      abbreviate      $\brwrange_\infty$      for      the      range
$\tbrw\infty{}(-\infty,\infty)$.  For  $\alpha\in\{c,\text{adj},+\}$, we
use     similar     notations    $\T_\alpha(i),     V_\alpha(i)$     for
$0\le i\le  \#\T_\alpha$, $i$  finite, with  $\T_\alpha(0)$ the  root of
$\T_\alpha$.   In particular,  $\brwrange_c=\tbrw  c{}[0,\#\T_c]$ is  in
agreement with \eqref{def-Rc}.   Finally, the law of  a branching random
walk started from $x$ is always denoted by $\mathbf P_x$.

Since  for all  $i\in \n$,  there exists  a unique  $j\in \n$  such that
$      \httree     -(i)=\T_\infty      (-j)$,     we      can     define
$\widehat V_-(i)=V_\infty (-j)$.  It is immediate that
\begin{equation*} \widehat \brwrange_-:= \widehat V_-[0, \infty)=V_\infty(-\infty, 0]=:\brwrange_-.\end{equation*}

\begin{figure}[t]
\includegraphics[height=6cm]{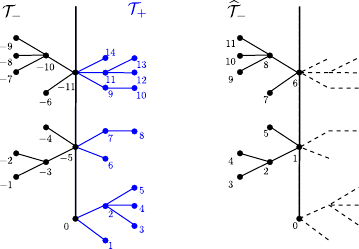}
\caption{\scriptsize
  On     the      left:     a      sample     of
  $\ttree-=\T_\infty (\z_  -)$ in black and  $\ttree+=\T_\infty (\n)$ in
  blue (except for  the root which is black but  belongs also to $\ttree+$)
  so that  $\T_\infty=\T_-\cup\T_+$; notice $\T_\infty(-1)$ is  coded by
  $(-1)(-1)(-1)$, $\T_\infty(-3)$  by $(-1)(-1)$ and  $\T_\infty(-5)$ by
  $(-1)$.  On the  right: the  tree  ${\httree-}$ which  is a  different
  ordering than $\ttree-$ of the same graph.}
\label{fig:tree_models}
\end{figure}

 The distribution of the random
tree $\T_\infty $ is invariant by rerooting at vertex $\T_\infty (n)$
for all $n\in \z$, see \cite[Section 2.2]{bai2020capacity}.  
 Since the distribution $\theta$ of the increments
of the branching random walk is symmetric, we directly deduce that  $\tbrw\infty {}$ is invariant by translation (\cite[Section 2.2]{bai2020capacity} and \cite{LeGall-Lin-range}):
\begin{align}\label{eq:translational_invariance}
\pars*{\tbrw\infty {}(n+i)-\tbrw\infty {}(n)}_{i\in\z}
\overset{\text{(law)}}{=}
\pars*{\tbrw\infty {}(i)-\tbrw\infty {}(0)}_{i\in\z},\quad\forall n\in\z.
\end{align}

Denote  by $(Y_n)_{n\in \n}$  the Lukasiewicz  walk associated to
${\ttree c}$,  that is,
 $(Y_n)$ is a
 centered random  walk  on $\z$  starting from $0$ 
 whose step distribution  is $\P(Y_1=i)=\mu(i+1)$,
$i\ge  -1$, and such that  $Y_{n+1}-Y_n+1$  is the
number of children of $\ttree c(n)$ for all $n<\#\T_c$.
 Similarly,  denote  by $(L_n)$  the  Lukasiewicz  walk  of
$\httree-$, then  its law is  as follows: $L_0=0,L_1=-1$; for  $n\ge 1$,  
conditioning on $\sigma\{L_i, i\le  n\}$, $L_{n+1}-L_n+1$ is distributed
as $\mu$  if $L_n\neq\min_{0\le i\le  n}L_i$, and as  $\widetilde\mu$ if
$L_n=\min_{0\le i\le
  n}L_i$. (To be precise the the Lukasiewicz  walk $(L_n)$ is associated to the forest in $\httree-$, where the edges of the infinite spine $(\varnothing_n)_{n\ge 0}$  are removed. In this setting, $\varnothing_{n+1}$ is not seen as child of $\varnothing_n$.)
By \cite[Section 5]{zhu2021critical}, ${\htbrw - {}}$ can be compared to
$\tbrw c {}$  in terms of $(Y_n)$  and $(L_n)$: for any $0\le  k< m$ and
nonnegative measurable function $F$,
\begin{equation} 
\label{eq:absolutecontinuity1} 
\E_x\Big( F(\tbrw c {}[0, k]) \, 
\big|\,   \#\ttree c= m\Big)=
 \E_x\Big( F(\htbrw - {}[0, k])   \,\Phi_{m,k}(L_k) \Big), \end{equation}
\noindent where 
\[\Phi_{m,k}(\ell):= \frac{m \P(Y_{m- k}= - (\ell +1))}{(m- k) \P(Y_m=-1)}\cdot\] 
 
By the local central limit theorem for the random walk $Y$, for any
$a\in (0,1)$, there exists some $C_a>0$ such that, for any $m\ge 1 $ and
$k\leq  am$,  we have 
$\Phi_{m,k}(\ell)\leq  C_a$ for all $\ell\in \z$. Consequently for
any   nonnegative measurable function $F$,  we have 
\begin{equation} 
\label{eq:absolutecontinuity2} 
\E_x\Big( F(\tbrw c {}[0, \lfloor a m\rfloor ]) \, 
\big|\,   \#\ttree c= m\Big) \le C_a \, 
\E_x\Big( F(\htbrw - {}[0, \lfloor a m\rfloor])  \Big).
\end{equation}

Finally, we  denote the spatial positions of the spine $\X$  
by \begin{equation} \label{def-X}
\begin{aligned}  
   \brwrange_\X &=\{V_\X(0), V_\X(1),  ...\} , 
   \end{aligned}
   \end{equation}
   where for  any $i\ge  0$, $V_\X(i):=  V_\infty(\varnothing_i) \in \z^d $ is  the 
   position of  $i$th point in the spine (recall that $\varnothing_0:=\varnothing$ is the root).  As a matter of  fact, the sequence
   $(V_\X(i))_{i\ge  0}$,  forms a  $\theta$-walk  on  $\z^d$. We denote by $\Y$ the set of points in $\widehat \T_-$ that are not in the spine, with the exception of the root,  see Fig.~\ref{fig:tree_XY}. 
   Let \begin{equation} \label{def-Y}
\begin{aligned}  \Y&:=\{\Y_0, \Y_1, \Y_2, ....\}= \httree-\backslash \{\varnothing_1, \varnothing_2, ...\},
\\   \brwrange_\Y&:= \{V_\Y(0), V_\Y(1), V_\Y(2), ....\}, 
\end{aligned} \end{equation} listed in depth-search order,  $\Y_0:=\varnothing$ and  $V_\Y(0):=V_\infty(\varnothing)$, and for any
$i\ge 0$, $V_\Y(i)=V_\infty(\Y_i)$ denotes the spatial position of the
$i$th vertex of $\Y$.  Then we have  $\httree-= \X\cup \Y$ and $\X\cap \Y= \{\varnothing\}$.

Let  $\brwrange_+:=V_+[0, \infty)$ denote the range of $V_+$. 
By construction, the random sets
$(\brwrange_\X,\brwrange_\Y)$ and
$(\brwrange_\X,\brwrange_+)$ conditionally on
$\{\T^\text{d}_0=\{\varnothing\}\}$
have the same distribution. 
  We deduce that 
\begin{align*}
\widehat \brwrange_-=\brwrange_\X\cup \brwrange_\Y\overset{\text{(law)}}{=}
\parsof*{\brwrange_\X\cup \brwrange_+}{ \T^\text{d}_0=\{\varnothing\}},
\end{align*}
and  that for any nonnegative measurable function $F$,
\begin{equation} \label{eq:t0_trick}
\E_x \Big( F(\brwrange_\Y, \Y)\Big)= \E_x \parsof*{ F(\brwrange_+,
  \ttree + )} {\T^\text{d}_0=\{\varnothing\}} \le \frac{1}{\mu(0)} \E_x \Big(
F(\brwrange_+, \ttree + )  \Big) . 
\end{equation}
This property will be used to compare rare events for $\htbrw - {}$ and $\tbrw + {}$. 
As for $V_\alpha[a, b]$ we use the notation: for any $0\le i \le j$,   \begin{equation}   V_\Y[i, j]:= \{V_\Y(i), V_\Y(i+1), ..., V_\Y(j)\}, \label{Vij}
 \end{equation}
 
\noindent and define for any $m\ge 1$,  \begin{equation}
  \brwrange_{\X, \widehat V_-}^{(m)}:= \brwrange_\X\cap(\htbrw-{}[0,m])
  \quad\text{and}\quad
  \brwrange_{\Y, \widehat V_-}^{(m)}:=\brwrange_\Y\cap(\htbrw-{}[0,m]) . \label{Vym}  \end{equation}

\begin{figure}
\includegraphics[height=6cm]{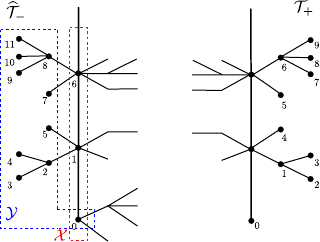}
\caption{\scriptsize An  illustration of $\httree-= \X\cup  \Y$, and its
  comparison  to $\ttree  +$ when  $
  \T_0^\text{d}=\{\varnothing\}$. We  have here that $\X(0)=0, \X(1)=1, \X(2)=6$ and 
  $\Y(0)=0, \Y(1)=\widehat   \T_-(2)$, $  \Y(2)=\widehat  \T_-(3)$,   $\Y(3)=\widehat
  \T_-(4)$, $\Y(4)= \widehat \T_-(5)$ and  $\Y(5)=\widehat \T_-(7)$. }
\label{fig:tree_XY}
\end{figure}

We end this section by a simple estimate on the number of spine points in $\httree -$.

\begin{lemma}\label{lem:easy_Z}
  Assume \eqref{hyp-tree}.  For any  $r \in (0, 1)$, there exist some positive constants $C=C_r$ and $c=c_r$ such that  for all   $n\ge 1$ 
\[
\P\pars*{\#(\httree -[0,n]\cap \X)> r \, n   } \le  C\, e^{-c \, n}.
\] 
\end{lemma}

\begin{proof}
  By construction,  $\httree -$ consists  of i.i.d.\ adjoint trees  in a
  sequence, and we need to show that it is very unlikely for $\httree -$
  to  cover $\lceil r n\rceil$  such  subtrees.  Denote  by  $(Z_i)$ i.i.d.\  random
  variables     distributed    as     the     total    population     of
  $\#\ttree{\text{adj}}$.  Then it suffices to show that
\[
\P\pars*{Z_1+\dots+Z_{\lceil r n\rceil -1} \le n} \le C\, e^{-c\, n},
\]

\noindent where this probability is equal to $1$ if $\lceil r n\rceil =1$. Let for $s\ge 0$, 
\[
f(s):=e^{s/r}\E\pars*{e^{-s Z_1}}.
\]
\noindent Note that $\E(Z_1)=\infty$.  Then 
$f'(s)=\E[(\frac1{r}-Z_1) e^{(\frac1{r}-Z_1)s}]$  
is continuous on $s\in(0,\infty)$ with $\lim_{s\rightarrow 0+}f'(s)=-\infty$.
Combined with $f(0)=1$, we know that $f(s_0)<1$ for some $s_0>0$. Therefore,
\[
\P\pars*{Z_1+\dots+Z_{\lceil r n\rceil-1} \le n}\le e^{s_0 n}\pars*{\E\pars*{e^{-s_0 Z_1}}}^{\lceil r n\rceil -1}\le \E\pars*{e^{-s_0 Z_1}}^{-1} \, \pars*{f(s_0)}^{r n},
\]
and we conclude by taking $c=- r \log f(s_0)>0$ and $C$ large enough.
\end{proof}

\subsection{Some known results. } \label{sub:known} We collect here some preliminary estimates on the random walk and the branching random walk.

\begin{lemma} \cite[Proposition 2.1.2, Theorem 2.3.9]{Lawler-Limic} \label{lem:2.1.2}
Let $(\xi_n)_{n\ge 0}$ be a centered, aperiodic and irreducible random
walk in $\z^d$ with finite variance. There exists some $C>0$ such that
for every $s>0$ and $n\geq 1$,
\begin{eqnarray}    \label{max-rw}
\rwp_0\pars*{\max_{0\le i\le n}|\xi_i|\ge s\sqrt n}&\le  & C\, s^{-2},
\\   
\sup_{x\in \z^d}\rwp_0(\xi_n=x)&\le & C\, n^{-d/2}. \label{rw-local}
 \end{eqnarray}
\end{lemma}

Indeed, 
 \cite[Proposition 2.1.2 (a)]{Lawler-Limic}, with $k=1$,  gives \eqref{max-rw}, and  \eqref{rw-local} follows from \cite[Theorem 2.3.9]{Lawler-Limic}.

Recall $\tbrw+{}$ is the branching random walk in $\z^d$ indexed by $\ttree +$ defined in Section \ref{sec:tadj}.

\begin{lemma} {\rm (\cite[Lemma 4.13]{bai2022convergence})}\label{lem:reversed_intersection}
 Let $d=5$.  Assume \eqref{hyp-tree},  \eqref{hyp-brw} with $q>4$,  and
 \eqref{hyp-xi}.
For any $M>0$, there exist  some $\upsilon, C>0$ such that for all $\varepsilon\in (0,1)$, 
\begin{equation}
\limsup_{n\to\infty}\P_0\pars*{\max_{|x|\le   \varepsilon n^{1/2}, x\in \z^5}  \rwp_x(\xi[0, \infty)\cap \tbrw+{}[0, n^2]=\emptyset) \ge \varepsilon^{\upsilon}} \le C \,  \varepsilon^{M}. \label{reversed_intersection}
\end{equation}
\end{lemma}
\begin{proof}
Though \cite{bai2022convergence} only considered simple random walks, a
general random walk $\xi$ that is centered with finite variance,
aperiodic and irreducible will validate every argument except for
\cite[Lemma 4.6]{bai2022convergence}, which relies on the asymptotics of
the  Green's function, $g_{(\xi)}(x)=\sum_{n\ge 0}\rwp_0(\xi_n=x).$
Given the existence of finite third moment of $\xi$, we can apply  
\cite[Theorem 2]{MR1625467} (with $N=5,m=0$) to get 
$g_{(\xi)}(x)=(C+o(1))|x|^{-3}$.
Thus we still get \cite[Lemma 4.6]{bai2022convergence} for $\xi$
satisfying~\eqref{hyp-xi}
instead of a simple random walk. 
\end{proof}

We end this subsection by some estimates  on the branching random walks
$V_+$ and $V_c$ and on the graph distance on $\T_+$. We consider $\T_+$
as a subgraph of $\T_\infty$ and, for all $0\le i<j$,   we    denote by ${\rm d_{gr}}(\T_+(i), \T_+(j))$    the graph distance in $\T_\infty$  between the two vertices $\T_+(i)$ and $\T_+(j)$, which is the number of edges in the geodesic path  connecting $\T_+(i)$ to $\T_+(j)$ in $\T_\infty$.

\begin{lemma}\label{lem:3.7}
Assume \eqref{hyp-tree} and \eqref{hyp-brw}  with some $q>4$.  

(i) For any $0< \zeta< \frac14 - \frac1{q} $,    there exists   some positive constants $a'$ and $C=C_{a'}$ such that  all $ n\ge1$ and $0<\eta<1$, we 
have  
\begin{equation}   
\P_0\Big( 
\max_{0 \le i< j \le n,\,  0\le  j -i\le \eta  n} |  \tbrw+{}(j)-   \tbrw+{}(i)| \ge \eta^\zeta n^{1/4}\Big) 
\le C  \, \eta^{a'}.  \label{increment-V+}
 \end{equation}

(ii) There exists some $C=C_q>0$ such that for any $n\ge 1$, \begin{equation}   \E_0\parsof*{ \max_{z\in \brwrange_c}|z|^q}{\#\ttree c=n} \le C\,  n^{q/4}. \label{max-Vc}
 \end{equation}

(iii)   For any $r\ge  1$, there is some positive constant $C$ such that for any $0\le i<j$,  \begin{equation}    \E\Big({\rm d_{gr}}(\T_+(i), \T_+(j))^{r} \Big) \le C \, (j-i)^{r/2}.
  \label{dTiTj}\end{equation}

 \end{lemma}
  \begin{proof} The statements (i) and (ii) come  from  \cite[(4.11), (4.24)]{bai2022convergence}.
 
  For (iii), we   deduce from the invariance  $\T_\infty $ by
  rerooting  
  that for any $j>i$, $   {\rm d_{gr}}(\T_+(i), \T_+(j)) $ is distributed as ${\rm d_{gr}}(\T_+(j-i), \varnothing) $. By using  \cite[(4.7), (4.9), (4.10)]{bai2022convergence}, and we note that (4.9) and (4.10)  hold  for any exponent $r \ge 1$ instead of $q/2$ there, we get \eqref{dTiTj}.
    \end{proof}

\subsection{Increments of \texorpdfstring{$V_+$}{}}\label{sub:increment}
 We first present a general result to estimate  the increments of a discrete-time process:
 
 \begin{lemma}\label{l:increment} Let $\alpha \in (0,1], b>0, p> \max(\frac1\alpha, \frac{b}{\alpha})$ and $ H, K\ge 0$ and $n\ge 1$. Let  $\QQ_0, ..., \QQ_n$ be real-valued random variables such that   for any $0\le i \le j \le n$, \begin{equation}   \e (|\QQ_j- \QQ_i|^p) \le \, H (j-i)^{\alpha p } + K (j-i)^b .  \label{GRR-1}   \end{equation}  For any $\max(1, b) < \gamma< \alpha p$,  there exists some positive constant $c=c(\alpha, \gamma, p, b)$ only depending on $\alpha, \gamma, p, b$ such that \begin{equation}    \e\Big( \max_{0\le i \le j \le n} |\QQ_j-\QQ_i|^p\Big) \le c\, \big( H n^{\alpha p} + K n^{\gamma}\big).  \label{GRR-2} \end{equation}
 \noindent In particular, if  $\QQ_0=0$, then \begin{equation}    \e\Big( \max_{0\le i  \le n} |\QQ_i|^p\Big) \le c \,\big( H n^{\alpha p} + K n^{\gamma}\big).  \label{GRR-3} \end{equation}
 \end{lemma}
 
 When $K=0$, the above lemma is the discrete-time version of the Garsia-Rodemich-Rumsey lemma. We shall apply Lemma \ref{l:increment} to estimate the increments of $V_+$ and  the term $K(j-i)^b$ will appear when we consider a truncated version of $V_+$, with $K$ depending on $n$. 
 
\begin{proof} Consider $(\QQ^{(n)}_t)_{0\le t \le 1}$ the linear interpolation of the process $(n^{-\alpha} \QQ_{\lfloor n t \rfloor})_{0\le t \le 1}$.  We are going to apply the Garsia-Rodemich-Rumsey lemma for  $\QQ^{(n)}$. We claim that  there exists some $c_p>0$ such that for all $0 \le s < t\le 1$,  \begin{equation}    \e \Big( |\QQ^{(n)}_t- \QQ^{(n)}_s|^p\Big) \le \, c_p  (t-s)^\gamma \big( H+ K n^{\gamma-\alpha p}\big).  \label{Thetant-s}\end{equation}

In fact, if for some $0\le i \le n$, $\frac{i}{n} \le s < t \le \frac{i+1}{n}$, then $\QQ^{(n)}_t-\QQ^{(n)}_s=  n^{1-\alpha}(t-s) (\QQ_{i+1}-\QQ_i)$, hence $$ \e \Big( |\QQ^{(n)}_t- \QQ^{(n)}_s|^p\Big) \le n^{p(1-\alpha)}(t-s)^p  (H+K)\le (H+K) \, n^{\gamma-\alpha p}\, (t-s)^\gamma ,  $$

\noindent proving \eqref{Thetant-s} (with $c_p\geq 1$) in this case as $\gamma< \alpha p$. 
If for some $0\le i < j< n$, $\frac{i}{n} \le s <\frac{i+1}{n}$ and
$\frac{j}{n} \le t <\frac{j+1}{n}$, we  distinguish two subcases: either
$j=i+1$, then we use twice the above estimate to the couples
$(s,\frac{i+1}{n})$ and $(\frac{i+1}{n}, t)$ and get that $$ \e \Big( |\QQ^{(n)}_t- \QQ^{(n)}_s|^p\Big) 
\le 2^p (H+K) \, n^{\gamma-\alpha p}\, (t-s)^\gamma;$$

\noindent or $j\ge i+2$, then $t-s\ge \frac1{n}$ and 
   \begin{eqnarray*}    && \e \Big( |\QQ^{(n)}_t- \QQ^{(n)}_s|^p\Big) 
   \\
   &\le&
    3^{p-1} \e \Big( |\QQ^{(n)}_t- \QQ^{(n)}_{j/n}|^p\Big) +  3^{p-1} \e \Big( |\QQ^{(n)}_{i/n}- \QQ^{(n)}_{s}|^p\Big) + 3^{p-1} \e \Big( |\QQ^{(n)}_{j/n}- \QQ^{(n)}_{i/n}|^p\Big)
   \\
&\le&  2\, 3^{p-1}  n^{-\alpha p }  (H+K) + 3^{p-1} n^{-\alpha p} \e(|\QQ_j- \QQ_i|^p)
\\
&\le&   c_p  H (t-s)^{\alpha p} + c_p K (t-s)^b n^{b- \alpha p}
\\
&\le&
c_p  (t-s)^\gamma \big( H+ K n^{\gamma-\alpha p}\big), 
\end{eqnarray*}
 
\noindent where in the last inequality we have used the facts that $b< \gamma< \alpha p$ and $(t-s)^{b-\gamma} \le n^{\gamma-b}$ in  this case. Then the proof of \eqref{Thetant-s} is complete. 

Now in view of \eqref{Thetant-s},  we are entitled to apply   \cite[(3a.3)]{Barlow-Yor} and get that for any (arbitrarily) fixed $0 < \delta < \gamma-1$, $$ \e\left(\sup_{0\le s < t \le 1} \frac{|\QQ^{(n)}_t-\QQ^{(n)}_s|^p}{|t-s|^\delta} \right) \le C \, \big( H+ K n^{\gamma-\alpha p}\big).$$

\noindent Hence, we get  $$\e\left(\max_{0\le i < j \le n}  |\QQ_j-\QQ_i|^p  \right)= n^{\alpha p}  \e\left(\sup_{0\le s < t \le 1}  |\QQ^{(n)}_t-\QQ^{(n)}_s|^p  \right) \le C \, \big( H n^{\alpha p} + K n^{\gamma}\big),$$

\noindent proving \eqref{GRR-2}. 
\end{proof}

We then apply this lemma to get a control on the increments of $V_+$. 

\begin{lemma}\label{l:max_diam2-V+}
Assume \eqref{hyp-tree} and \eqref{hyp-brw} with some $q>4$. For any  $p\ge \frac{4}{\frac14- \frac{1}{q}}$,  there exists  some $ C>0$ such that for any  $\eta \in (0, 1)$ and $r>0$,  we have   \footnote{The term  $ k\eta  n$ is understood as its integer part. Similar remark applies elsewhere without further explanations.} 
\begin{equation}\label{eq:max_diam2-V+} 
\limsup_{n\to\infty} \P_0\Big( 
\max_{0\le k \le  \frac1\eta} \max_{0\le i \le \eta  n} | V_+(i+ k  
 \eta  n)- V_+(k \eta  n )| \ge  r\,  n^{1/4} \Big) 
\le C    \, r^{-p} \, \eta^{\frac{p}{4}-1}.
 \end{equation}
 Consequently, \begin{equation}\label{eq:max_diam2-Y++} 
\limsup_{n\to\infty} \P_0\Big( 
\max_{0\le k \le  \frac1\eta} \max_{0\le i \le \eta  n} | V_\Y(i+ k  
 \eta  n)- V_\Y(k \eta  n )| \ge  r\,  n^{1/4} \Big) 
\le \frac{C}{\mu(0)}    \, r^{-p} \, \eta^{\frac{p}{4}-1}.
 \end{equation}
 \end{lemma}

 Typically    we    will    choose     $r=    \eta^\zeta$    for    some
 $\zeta\in   (0,  \frac14)$   and   $p$  large   enough   so  that   the
 right-hand-side of  \eqref{eq:max_diam2-V+} goes to $0$  as $\eta\to0$.
 Comparing  such a  choice  of $r$  in  Lemma \ref{l:max_diam2-V+}  with
 \eqref{increment-V+},     we     remark      that     the     condition
 $\zeta< \frac14  - \frac1{q}$ is relaxed to $\zeta<1/4$ 
 at   the    expense   of   the    limit   as   $n   \to    \infty$   in
 \eqref{eq:max_diam2-V+}. This relaxation is crucial for the forthcoming
 Lemma \ref{l:neighbor}, where we must select $\zeta$ sufficiently close
 to $\frac14$ (in fact we need  $\zeta > \frac1{d}$) to control the size
 of a neighborhood of $V_+$ or of $\widehat V_-$.

 \begin{proof}  It   is  enough  to  show   \eqref{eq:max_diam2-V+},  as
   \eqref{eq:max_diam2-Y++}  follows   from  \eqref{eq:max_diam2-V+}  by
   \eqref{eq:t0_trick}. To this end, we  first use a truncation argument
   (Step   1)  to   remove  the   big   jumps  in   $V_+$,  then   Lemma
   \ref{l:increment}  (Step  2)  to   estimate  the  increments  of  the
   truncated version of $V_+$.

{\bf Step 1:}
Set
\[
b_n:=n^{\frac 1 8+\frac 1 {2q}}.
\]
Recall that on a branching random walk, $X_e$ denotes the displacement on an edge $e$; we denote
\[
X^{(n)}_e:=X_e\,  1_{\{|X_e|<b_n\}},
\]
and we write $V_{+,n}$ for the corresponding branching random walk where displacements on edges are $(X^{(n)}_e)$ instead of $(X_e)$. Our first step is to show that we may replace $\tbrw{+}{}[0,n]$ by $\tbrw{+,n}{}[0,n]$, i.e.  
\begin{align}\label{eq:no_large_jump}
\limsup_{n\rightarrow\infty}\P\pars*{\tbrw{+,n}{}[0,n]=\tbrw+{}[0,n]}=1.
\end{align}

 Denote by ${\tt t}_n$ the subtree of $\T_\infty$ spanned by $\T_+[0,n]$ and 
  $\Delta_n:=\#({\tt t}_n\cap \X)$ the number of intersections of ${\tt t}_n$ with the spine $\X$.  Note that  the subtree of $\T_\infty$ spanned by $\T_+\backslash \T_0^\text{d}$ and the root $\varnothing$, in depth-first order,  has the same distribution as $\widehat \T_-$.  Since  $\T_+\backslash \T_0^\text{d}$ is independent of $\T_0^\text{d}$, we get that for any $1\le j \le n+1$ and $k\ge 1$, $\p(\Delta_n =  k | \#\T_0^\text{d}=j) = \p(\#(\widehat \T_-[0, n+k-j] \cap \X) = k)\le \p(\#(\widehat \T_-[0, n+k] \cap \X) \ge  k).$ It follows that $\p(\Delta_n =  k)= \sum_{j=1}^{n+1} \p(\Delta_n =  k , \,  \#\T_0^\text{d}=j) \le  \p(\#\widehat \T_-[0, n+k] \cap \X) \ge  k)$, hence  \begin{eqnarray*}    \p(\Delta_n\ge n) 
  &\le& \sum_{k=n}^\infty \p(\#(\widehat \T_-[0, n+k] \cap \X) \ge  k) 
  \\
  &\le& \sum_{k=n}^\infty \p\Big(\#(\widehat \T_-[0, n+k] \cap \X) \ge  \frac{n+k}{2}\Big) 
  \\
  &\le& \sum_{k=n}^\infty e^{-c (n+k) } \le C \, e^{-c \, n},
 \end{eqnarray*}

\noindent where the third inequality is due to Lemma \ref{lem:easy_Z}
(with $r=1/2$ there). Note that there are $n+\Delta_n$ edges in ${\tt
  t}_n$. By union bounds and \eqref{hyp-brw}, we get
$$
 \P\Big(\max_{e\in {\tt t}_n}|X_e|\ge b_n\Big)
\le  2n\P(|X|\ge b_n)+ C \, e^{-c n } \le C  n b_n^{-q}+  C \, e^{-c n }
\le
C'\, n^{\frac 1 2-\frac q 8}
$$
 
\noindent for all large $n$, where $X$ denotes a random variable with distribution $\theta$. 
In other words, with probability larger than $1-C'\, n^{\frac 1 2-\frac q 8}$, there is no edge $e\in {\tt t}_n$ such that $X_e\ne X^{(n)}_e$, hence $\tbrw{+,n}{}[0,n]=\tbrw+{}[0,n]$ and we get   \eqref{eq:no_large_jump}. 
Then  it suffices to prove \eqref{eq:max_diam2-V+}  for $V_{+,n}[0,n]$ instead of $\tbrw+{}[0,n]$.

 \medskip
{\bf Step 2}: 
  For $n\ge j>i\ge 0$, 
recall that ${\rm d_{gr}}(\T_+(i), \T_+(j))$ denotes  the graph distance in $\T_\infty$ between the two vertices $\T_+(i)$ and $\T_+(j)$.    Conditioning on $\{{\rm d_{gr}}(\T_+(i), \T_+(j))=k\}$,  $V_{+,n}(j)- V_{+, n}(i)$ is the sum of $k$ iid copies of $X^{(n)}:= X 1_{\{|X|\le b_n\}}$, where $X$ is distributed as $\theta$. As $X^{(n)}$ is centered (by the symmetry of $\theta$), we use the Rosenthal inequality  \cite[Theorem 2.9]{Petrov}, and get that for any $p>1$, there exists some $c_p>0$ such that $$\E\Big( |V_{+,n}(j)- V_{+, n}(i)|^p \, \big|\,  {\rm d_{gr}}(\T_+(i), \T_+(j))=k\Big) \le c_p \big( k \E(|X^{(n)}|^p ) +   k^{p/2} \E(|X^{(n)}|^2)^{p/2}\big).$$

\noindent Note that $ \E(|X^{(n)}|^p ) \le b_n^p $ and $ \E(|X^{(n)}|^2) \le \E(|X|^2)<\infty$.  Then by \eqref{dTiTj}, there is some positive constant $C$ such that for any $0\le i < j\le n$ and $p>4$,   $$
\E\Big( |V_{+,n}(j)- V_{+, n}(i)|^p  \Big)
\le 
C\, \big( b_n^p (j-i)^{1/2} + (j-i)^{p/4}\big).$$

\noindent    Applying \eqref{GRR-3} to $\Upsilon_\cdot=V_{+, n}(i+\cdot)$, $\alpha=\frac14$, $b=\frac12$ and   any  $\gamma \in (1, p/4)$, we get some positive constant $C'$ such that for any $0\le i < \ell\le n$, \begin{equation}   
\E\Big( \max_{i\le j \le \ell} |V_{+,n}(j)- V_{+, n}(i)|^p  \Big)
\le 
C'\, \big( (\ell-i)^{p/4} + b_n^p (\ell-i)^\gamma\big). \label{maxV+n}
 \end{equation}

This, in view of the union bound  and the  Markov inequality,  yields that \begin{eqnarray*} && \P_0\Big( 
\max_{0\le k \le  \frac1\eta} \max_{0\le i \le \eta  n} | V_{+, n}(i+ k  
 \eta  n)- V_{+, n}(k \eta  n )| \ge r\,  n^{1/4} \Big) 
 \\
 &\le&
C'\, \eta^{-1} \,  r^{-p} \,  n^{-p/4} \big( \eta^{p/4} n^{p/4}+ b_n^p \eta^{\gamma} n^\gamma\big) 
 \\
 &=& C' \, r^{-p} \, \eta^{p/4 - 1} + C'\, r^{-p}\, \eta^{ \gamma-1 } \,  b_n^p n^{-p/4 +\gamma}. \end{eqnarray*}

\noindent Now for any $p\ge  \frac{4}{\frac14- \frac{1}{q}}$, 
 we fix an arbitrary $\gamma\in (1, 2)$ so that $(\frac18+ \frac1{2q})p - \frac{p}{4} + \gamma <0$, 
then $b_n^p n^{-p/4 +\gamma}\to 0$ as $n\to \infty$, and we get \eqref{eq:max_diam2-V+} for $V_{+, n}$, hence for $V_+$ by Step 1. 
 \end{proof}

The following lemma says that with high probability the number of points in the neighbor of $\htbrw-{}[0,m]$ is not too big, such an estimate  will be important in estimating the forthcoming probability term $p'_{m, n}$ defined in \eqref{p'mn}.

\begin{lemma}\label{l:neighbor} Let $d\ge 5$. Assume  \eqref{hyp-tree} and \eqref{hyp-brw}  with $q> 4$.  For any $M>0$ and $\gamma\in (0, d-4)$, there exist  some positive constants  $ C=C_d$ and $C'= C_{M, \gamma}$ such that for any   $\eta \in (0, 1)$, \begin{equation}\label{eq:max_diam5}
 \limsup_{m\to\infty}  \P_0\Big( 
\#\Big( \big(\htbrw-{}[0,m]\big)^{\eta m^{1/4}}\cap \z^d\Big)
>
  C\,  \eta^\gamma m^{d/4}\Big) 
\le   C'\,   \eta^{M}  .
\end{equation}
\end{lemma}

\begin{proof}  Let $\zeta := \frac1{d-\gamma}\in (\frac1{d}, \frac14)$.  Applying
  \eqref{eq:max_diam2-Y++} (with $\eta$ replaced by $\eta^{1/\zeta}$ and
  $r$ by $\eta$) to a sufficiently large $p$, we get that for any $M>1$, there is some positive constant $C$  such that for all  $\eta \in (0, 1)$, 
\begin{equation}\label{eq:max_diam2}
\limsup_{m\to\infty} \P_0\Big( 
\max_{0\le k \le  \eta^{-1/\zeta}} \max_{0\le j \le \eta^{1/\zeta} m} | V_\Y(j+ k  
 \eta^{1/\zeta} m)- V_\Y(k \eta^{1/\zeta} m )| \ge \eta  m^{1/4} \Big) 
\le C    \, \eta^{M}.
\end{equation}

\noindent Recall  the notations  $\brwrange_{\X, \widehat V_-}^{(m)}=\brwrange_\X\cap\htbrw-{}[0,m]$, $\brwrange_{\Y, \widehat V_-}^{(m)}=\brwrange_\Y\cap\htbrw-{}[0,m]$ and  $V_\Y[i, j]:= \{V_\Y(i), V_\Y(i+1), ..., V_\Y(j)\}$.  Since $\brwrange_{\Y, \widehat V_-}^{(m)} \subset V_\Y[0,m]$,  we deduce from \eqref{eq:max_diam2} that
\begin{align}
1-\P_0\pars*{\brwrange_{\Y, \widehat V_-}^{(m)}\text{ can be covered by }\eta^{-{1/\zeta}}\text{ balls of radius }\eta m^{1/4}}
\le C \, \eta^{M} +o_m(1), \label{eq:max_diam2bis}
\end{align}

\noindent where $o_m(1) \to 0$ as $m\to\infty$ and $o_m(1)$ may depend on $\eta$.

Let $j\ge 1$. Consider the event $\{\dd(x, \brwrange_{\Y, \widehat V_-}^{(m)})\ge j\}$ with $x \in \brwrange_{\X, \widehat V_-}^{(m)}$. There is some $i$ such that $x=V_\infty(\varnothing_i)$. 
Let 
$K_i:=\dd_{gr}(\varnothing_i,\Y\cap\widehat {\mathcal T}_-[0,m])$, where
we recall that $\dd_{gr}$ is the graph distance in $\T_\infty $. Then $\dd(x, \brwrange_{\Y, \widehat V_-}^{(m)})$ is less than the displacement of a $\theta$-walk at time $K_i$.
Since each spine point has no attached tree with probability $\widetilde \mu(0)=1-\mu(0)$, we deduce from the union bound that for any $j, \ell\ge 1$,\begin{eqnarray} \P_0\Big( \max_{x\in \brwrange_{\X, \widehat V_-}^{(m)}}\dd(x, \brwrange_{\Y, \widehat V_-}^{(m)})\ge j\Big)
&\le&
\P_0\Big(\max_{0\le i \le m-1} K_i \ge \ell \Big) +  m \P_0(|S_\ell| \ge j)
\nonumber
\\
&\le&
m (1-\mu(0))^\ell +   m j^{-q'} \E_0(|S_\ell|^{q'}),
    \end{eqnarray}
\noindent where we choose (and then fix) an arbitrary constant $q'\in (4, q)$. 
By \cite[Theorem 2.10]{Petrov},  $$ \E_0(|S_{\ell}|^{q'}) \le C\, \ell^{q'/2}.
$$

\noindent 
 Take $\ell=\lfloor \frac{q'}{\mu(0)} \log m \rfloor$  and $j= \eta m^{1/4}$. There exists some $m_0=m_0(q',   \mu(0))\ge 1$ such that for all $m\ge m_0$,    \begin{equation} \P_0\Big( \max_{x\in \brwrange_{\X, \widehat V_-}^{(m)}}\dd(x, \brwrange_{\Y, \widehat V_-}^{(m)})\ge \eta m^{1/4}  \Big) 
 \le     C'  \eta^{-q'} \, m^{1- q'/4} (\log m)^{q'/2}.
   \label{distanceXY}\end{equation}

 \noindent   
 Note that on $\{\max_{x\in \brwrange_{\X, \widehat V_-}^{(m)}}\dd(x, \brwrange_{\Y, \widehat V_-}^{(m)})< \eta m^{1/4} \}$, if $\brwrange_{\Y, \widehat V_-}^{(m)}$ can be covered by $\eta^{-{1/\zeta}}$  balls of radius $\eta m^{1/4}$, then $\htbrw-{}[0,m]=\brwrange_{\X, \widehat V_-}^{(m)}\cup \brwrange_{\Y, \widehat V_-}^{(m)}$  can be covered by $\eta^{-{1/\zeta}}$  balls of radius $2\eta m^{1/4}$. 
  It follows from \eqref{eq:max_diam2bis} and \eqref{distanceXY} that 
\begin{eqnarray*}   && 1-\P_0\pars*{
\htbrw-{}[0,m]\text{ can be covered by }\eta^{-{1/\zeta}}\text{ balls of radius }2\eta m^{1/4}}
\\
&\le& C\,    \eta^{M} +   C'  \eta^{-q'} \, m^{1- q'/4} (\log m)^{q'/2}+ o_m(1)
\\
&=& C\,    \eta^{M}  + o_m(1).
\end{eqnarray*}

Finally, when $\htbrw-{}[0,m]$ is covered by $\eta^{-{1/\zeta}}\text{ balls of radius }2\eta m^{1/4}$, its $(\eta m^{1/4})$-neighborhood is covered by $\eta^{-{1/\zeta}}\text{ balls of radius }3\eta m^{1/4}$,
hence
\[\#\Big( \big(\htbrw-{}[0,m]\big)^{\eta m^{1/4}}\cap \z^d\Big)\le \eta^{-{1/\zeta}} \,  c_d\,\eta^d \, m^{d/4}= c_d\, \eta^\gamma\, m^{d/4},
\]

\noindent for some positive constant $c_d$ only depending on $d$. This proves \eqref{eq:max_diam5}.
\end{proof}

\subsection{Intersection probabilities between \texorpdfstring{$\tbrw\pm{}$}{} and \texorpdfstring{$\xi$}{}}
\label{sec:3.2}

The proof of Proposition \ref{p:intersection} can be outlined as follows. By ordering the vertices of $\T_c$  using the depth first search and its reversed sense  (see \eqref{eq:j0}), we may reduce the problem of studying $I(\varepsilon, n)$, defined in \eqref{def-Iepsilon}, to that of  the non-intersection probability $$\mathbf P_0\otimes \rwp_{x_n} 
\pars*{ \tbrw c{}\bracks*{0,\frac 3 5m} \cap (\xi[0,n])^{\varepsilon n^{1/2}} \not=\emptyset, \tbrw c{}\bracks*{0,{\frac 4 5}m} \cap \xi[0,n] =\emptyset \, \Big|\, \#\ttree c=m} ,
$$

\noindent where the choice $(\frac35, \frac45)$ can be replaced by any $(c, c+\lambda)$ with $c \in (\frac12, 1)$ and $\lambda\in (0, 1-c)$. 
Furthermore,  the main contribution to $I(\varepsilon, n)$ comes from those $m$ of order $n^2$, say $m \in [ \delta n^2, \frac1{\delta} n^2]$ for small $\delta>0$. By \eqref{eq:absolutecontinuity2}, the above probability is dominated, up to a multiplicative constant,  by the corresponding probability for $\widehat V_-$ instead of $V_c$ conditioned on $\{\#\ttree c=m\}$. This latter probability is estimated  in the following result.

\begin{proposition} \label{l:p(m,n)} Let $d=5$.  Assume  \eqref{hyp-tree},  \eqref{hyp-brw}  with $q> 4$ and \eqref{hyp-xi}. Fix $\lambda>0$. 
For each $\varepsilon>0$, $x\in \r^5$  and $x_n=[x\sqrt n] \in \z^5 $, we define
\[ 
p_{m, n}  :=  \mathbf P_0\otimes \rwp_{x_n} \pars*{ \htbrw -{}[0, m] \cap (\xi[0,n])^{\varepsilon n^{1/2}} \not=\emptyset,  \htbrw -{}[0, (1+\lambda) m] \cap \xi[0,n] =\emptyset }.
\]
There is  some $c>0$ such that for  any $\delta\in (0, \frac19)$, there is   $C=C_{\delta, x, \lambda}>0$ such that  for any $ \varepsilon \in (0,1)$,  
\[
\limsup_{n\to\infty}\max_{m \in [ \delta n^2, \frac1{\delta} n^2]} p_{m, n} \le C  \,     \varepsilon^c .   
\]
\end{proposition}

This subsection is devoted to the proof of Proposition \ref{l:p(m,n)}. Consider   $\varepsilon \in (0, 1)$. Let $r>0$ be some small constant whose value will be determined later. We have \[
p_{m, n} \le p'_{m, n} + p''_{m, n},
\]

\noindent with 
\begin{eqnarray}     p'_{m, n}
& :=  &
\P_0\otimes \rwp_{x_n} \Big( \htbrw-{}[0, m] \cap \xi[(1-\varepsilon^r) n , n ]^{\varepsilon n^{1/2}} \not=\emptyset\Big), \label{p'mn}
\\
p''_{m, n}&:=&  \P_0\otimes \rwp_{x_n} \Big( \htbrw-{}[0, m] \cap \xi[0, (1-\varepsilon^r)n ]^{\varepsilon n^{1/2}} \not=\emptyset,   \nonumber
\\
&& \qquad \htbrw-{}[0, (1+\lambda)m] \cap \xi[0, n]=\emptyset \Big). \label{p''mn}
\end{eqnarray}

We outline the steps in estimating  $p'_{m, n}$ and $p''_{m, n}$ as follows. 
\medskip

{\bf (i) Estimate $p'_{m, n}$.}  Let $r' \in (0, r)$. According to the
random walk estimate \eqref{max-rw}, with high probability as
$\varepsilon\rightarrow 0$, $\xi[(1-\varepsilon^r)n,n]$ is contained  in $\ball(\xi_n,\varepsilon^{r'/2} n^{1/2}) $ the ball centered at $\xi_n$ with a radius of $\varepsilon^{r'/2} n^{1/2}$. Therefore, estimating $p'_{m, n}$ boils down to evaluating whether $\xi_n$ is in $\big(\htbrw-{}[0,m]\big)^{\eta n^{1/2}}$, where $\eta$ is some power of $\varepsilon$. This is the motivation for estimating $ \#\big(\htbrw-{}[0,m]\big)^{\eta n^{1/2}}$, a task addressed in Lemma \ref{l:neighbor}.  
 
\medskip
{\bf (ii) Estimate $p''_{m, n}$.} We divide $\httree-$ into two pieces $\X$ and $\Y$, as done in \eqref{def-X} and \eqref{def-Y}, and give an estimate on their ranges, see Fig.~\ref{fig:tree_XY}.    By Lemma \ref{lem:easy_Z}, there are few points in $V_\X\cap \htbrw-{}[0, m] $, then  we may replace  $\htbrw-{}$ by $V_\Y$ in $p''_{m, n}$.  This, in view of \eqref{eq:t0_trick},   boils down to estimating the corresponding probability for $V_+$. To achieve this, we will cut the range of $V_+$ into a certain number $\lfloor \varepsilon^{-1/\zeta}\rfloor $ of pieces. An application of the strong Markov property of $\xi$ allows us to reduce the problem to estimating the expectation of the maximum over $\lfloor \varepsilon^{-1\zeta}\rfloor$ of identically distributed random variables,  whose common law is that of $q_{\varepsilon, r}(n)$, where  
\begin{equation}    q_{\varepsilon, r}(n):= \max_{|x|\le  4 \varepsilon n^{1/2}, \, x\in \z^d}  \rwp_x(\xi[0,  \varepsilon^r n]\cap \tbrw+{}[0, \varepsilon n^2]=\emptyset)   1_{\{\tbrw+{}[0,\varepsilon  n^2]\subset\ball(0,\varepsilon^\zeta n^{1/2})\}},  \label{def:qrepilson}\end{equation}

\noindent with $\zeta\in (0, \frac14-\frac1{q})$ a fixed constant.  We mention that  $\rwp_x$ only computes the probability with respect to the random walk $\xi$ starting from $x$, so that $q_{\varepsilon, r}(n)$  is a {\it random variable} depending on $V_+$.  Finally, $q_{\varepsilon, r}(n)$ is estimated in the following Lemma \ref{l:xi-new}.

 \medskip

\begin{lemma} \label{l:xi-new} Let $d=5$.  Assume  \eqref{hyp-tree},  \eqref{hyp-brw}  with $q>4$ and \eqref{hyp-xi}.    Let $\zeta\in (0, \frac14-\frac1{q})$  and $q_{\varepsilon, r}(n)$ be as in \eqref{def:qrepilson} and $r\in (0, 2 \zeta)$. For any $M>0$, there exist  some $\upsilon=\upsilon_{r,\zeta}>0$ and $C=C_{r, \zeta}>0$ such that for all  $\varepsilon\in (0, 1)$,  
\begin{equation}   
\limsup_{n\to\infty}\P_0(q_{\varepsilon, r}(n) \ge \varepsilon^\upsilon) \le C\,  \varepsilon^M.\label{eq:qrepsilon}
 \end{equation} 
\end{lemma}

\begin{proof} Only small $\varepsilon$ needs to be considered. 
Apply a change of variables 
$(n',\varepsilon')=(\varepsilon^{1/2}n, 4\varepsilon^{3/4})$ to
\eqref{reversed_intersection} in Lemma~\ref{lem:reversed_intersection}, we deduce that there are some $\upsilon', C', C''>0$ such that
\begin{align}\label{eq:known_ST}
\limsup_{n\rightarrow \infty } \P_0\pars*{\max_{|x|\le   4 \varepsilon n^{1/2}, x\in \z^5}  \rwp_x(\xi[0, \infty)\cap \tbrw +{}[0,  \varepsilon n^2]=\emptyset) \ge C' \varepsilon^{\upsilon'}} \le C'' \varepsilon^{M}.
\end{align}

Let $\ell_n=\varepsilon^\zeta n^{1/2}$. On the event $\{\tbrw+{}[0,\varepsilon  n^2]\subset\ball(0,\ell_n)\}$, we deduce from the Markov property of $\xi$ at $\varepsilon^r n$ that \begin{eqnarray*}
 &&
 \rwp_x(\xi[0, \infty)\cap \tbrw+{}[0,\varepsilon  n^2]=\emptyset)
 \\
 &\ge& 
 \rwp_x\Big(
\xi[0,\varepsilon^{r} n]\cap \tbrw+{}[0,\varepsilon n^2]=\emptyset,\,|\xi_{\varepsilon^{r}n}|>2\ell_n,
\,\xi[\varepsilon^{r}n,\infty)\cap \ball(0,\ell_n)=\emptyset\Big)
\\
&\ge&
\inf_{|y|\ge 2 \ell_n}\rwp_y\pars*{\xi[0,\infty)\cap \ball(0,\ell_n)=\emptyset}\,\,  \rwp_x\pars*{
\xi[0,\varepsilon^{r}n]\cap \tbrw+{}[0,\varepsilon n^2]=\emptyset,
\,|\xi_{\varepsilon^{r}n}|>2\ell_n}
\\
&\ge& c \,\rwp_x\pars*{
\xi[0,\varepsilon^{r}n]\cap \tbrw+{}[0,\varepsilon n^2]=\emptyset,
\,|\xi_{\varepsilon^{r}n}|>2\ell_n},
\end{eqnarray*}

\noindent for all $n$ such that $\ell_n \ge R_0$, where $R_0>0$ is a large enough constant such that  
\begin{align*}
c=\inf_{R\ge R_0}\inf_{|y|>2R}\rwp_y\pars*{\xi[0,\infty)\cap \ball(0,R)=\emptyset}>0.
\end{align*}

\noindent By the local limit theorem for the random walk \eqref{rw-local}, there is $C'''>0$ such that \begin{align*}
\max_{|x|\le 4\varepsilon n^{1/2}}\rwp_x(|\xi_{\varepsilon^{r}n}|\le 2\ell_n)
\le \rwp_0(|\xi_{\varepsilon^{r}n}|\le 3\ell_n)
\le C''' \varepsilon^{\zeta d - rd/2}.
\end{align*}

\noindent It follows that $$q_{\varepsilon, r}(n) 
\le
\frac1{c}  \max_{|x|\le   4 \varepsilon n^{1/2}}  \rwp_x(\xi[0, \infty)\cap \tbrw +{}[0,  \varepsilon n^2]=\emptyset) +  C''' \varepsilon^{\zeta d - rd/2}.$$

\noindent By \eqref{eq:known_ST}, we obtain Lemma \ref{l:xi-new} by choosing any $\upsilon\in (0, \min(\upsilon',  (\zeta   - r /2)d)).$ 
\end{proof}

We are now ready to present the proof of Proposition \ref{l:p(m,n)}.
\begin{proof}[Proof of Proposition \ref{l:p(m,n)}] As   in Lemma \ref{l:xi-new}, let $ \zeta\in (0, \frac14-\frac1{q})$ and fix some constant $r \in (0, 2\zeta)$. Recall \eqref{p'mn} and \eqref{p''mn}.  Let $\delta\in (0, \frac19)$.  It is enough to show that there is  some $c>0$ independent of $\delta$, and   some  constant $C_\delta>0$ such that  for any small $ \varepsilon >0$,  
\begin{eqnarray}    
\limsup_{n\to\infty}\max_{m \in [ \delta n^2, \frac1{\delta} n^2]} p'_{m, n} &\le&  C_\delta  \,     \varepsilon^c .   \label{estimate:p'} 
\\
\limsup_{n\to\infty}\max_{m \in [ \delta n^2, \frac1{\delta} n^2]} p''_{m, n} &\le&  C_\delta  \,     \varepsilon^c .   \label{estimate:p''}
\end{eqnarray}

We mention that the exact values of $C_\delta, c$ are not important, they may change from one paragraph to another. 

\medskip
 {\bf (i) Proof of \eqref{estimate:p'}.} 
Fix some  $0<r'<r$. Apply \eqref{max-rw} with $(n,s)=(\varepsilon^r n,\varepsilon^{(r'-r)/2})$ there, we have
\begin{equation*}
\rwp_{x_n}\pars*{\max_{(1-\varepsilon^r) n  \le k \le n } |\xi_k-\xi_n| \ge \varepsilon^{r'/2} n^{1/2} } \le C\,  \varepsilon^{r-r'}   . 
\end{equation*}

 Since $\varepsilon \le \varepsilon^{r'/2}$, we deduce  that for all $m \in [ \delta n^2, \frac1{\delta} n^2]$,   \begin{equation} 
 p'_{m, n}
  \le    C\, \varepsilon^{r-r'} + \P_0\otimes \rwp_{x_n} \Big( \xi_n \in \big(\htbrw-{}[0, m]\big)^{2\varepsilon^{r'/2}n^{1/2}}\Big) . \label{p'-1}
 \end{equation}

To estimate the above probability term, we   apply \eqref{eq:max_diam5} to $m=n^2/\delta$, $\eta= 2 \delta^{1/4} \varepsilon^{r'/2}   $  with $d=5$, $\gamma\in (0,1)$ and $M=1$,  to see that there are some $C_\delta, c_\delta>0$  (we may choose $c_\delta:= C 2^\gamma \delta^{-(d-\gamma)/4}$ and $C_\delta:= 2 C' \delta^{1/4}$ with $C, C'$ as in \eqref{eq:max_diam5})   such that   \begin{eqnarray}     && \limsup_{n\to\infty} \max_{m \in [ \delta n^2, \frac1{\delta} n^2]}\P_0\Big( 
\#\big((\htbrw-{}[0,m])^{2\varepsilon^{r'/2}n^{1/2}} \cap \z^5\big)
>
  c_\delta \, \varepsilon^{\gamma r'/2}n^{5/2}\Big) 
  \nonumber \\
  &\le& \limsup_{n\to\infty} \P_0\Big( 
\#\big((\htbrw-{}[0,\frac1{\delta} n^2])^{2\varepsilon^{r'/2}n^{1/2}} \cap \z^5\big)
>
  c_\delta \, \varepsilon^{\gamma r'/2}n^{5/2}\Big)
  \nonumber
  \\
&\le&  C_\delta \varepsilon^{  r'/2}    .
 \label{p'-2}
\end{eqnarray}

\noindent  Moreover, by \eqref{rw-local}, for any finite set $A\subset\z^5$,
\begin{equation*}
\rwp_{x_n}(\xi_n\in A) \le C\,  n^{-5/2}\#A.
\end{equation*}

\noindent It follows that  on  $\{\#\big((\htbrw-{}[0,m])^{2\varepsilon^{r'/2}n^{1/2}} \cap \z^5\big)
\le 
  c_\delta \, \varepsilon^{\gamma r'/2} \, n^{5/2}\}$, \begin{eqnarray*}  \rwp_{x_n} \Big( \xi_n \in \big(\htbrw-{}[0, m]\big)^{2\varepsilon^{r'/2}n^{1/2}}\Big)   
  \le
   C\, c_\delta\, \varepsilon^{\gamma r'/2}, \end{eqnarray*}

\noindent which in view of \eqref{p'-1} and \eqref{p'-2} yield that   for all $m \in [ \delta n^2, \frac1{\delta} n^2]$,   $$ 
 p'_{m, n}
\le  C\, \varepsilon^{r-r'} +  C_\delta \varepsilon^{ r'/2}      +   C\, c_\delta\, \varepsilon^{\gamma r'/2} + o_n(1),$$    
    \noindent with $o_n(1)$ independent of $m$ and $o_n(1) \to 0$ as $n\to\infty$. This proves \eqref{estimate:p'}.
\medskip

{\bf (ii) Proof of \eqref{estimate:p''}.} Recall \eqref{p''mn} for the definition of $p''_{m,n}$. By  \eqref{Vym},  $ \widehat  V_-[0, m]= \brwrange_{\X, \widehat V_-}^{(m)}\cup \brwrange_{\Y, \widehat V_-}^{(m)}$. We are going to apply \eqref{distanceXY} with $\eta= \varepsilon \delta^{1/4}$. Note that for all $m \in [ \delta n^2, \frac1{\delta} n^2]$, on the event $\{\max_{x\in \brwrange_{\X, \widehat V_-}^{(m)}}\dd(x, \brwrange_{\Y, \widehat V_-}^{(m)})< \varepsilon \delta^{1/4} m^{1/4} \}$, $\htbrw-{}[0, m] \cap \xi[0, (1-\varepsilon^r)n ]^{\varepsilon n^{1/2}} \not=\emptyset$ implies that $\brwrange_{\Y, \widehat V_-}^{(m)}\cap \xi[0, (1-\varepsilon^r)n ]^{2\varepsilon n^{1/2}} \not=\emptyset$. It follows from  \eqref{distanceXY} that  for all $m \in [ \delta n^2, \frac1{\delta} n^2]$ and $n$ large enough, 
\begin{equation}  p''_{m, n} 
\le p'''_{m, n}  + o_n(1),  \label{def-p'''mn} \end{equation} 

\noindent where as before, $o_n(1)$ is independent of $m$ and $o_n(1)\to1$ as $n\to\infty$, and  $$   p'''_{m, n}   :=  \P_0\otimes \rwp_{x_n} \Big( \brwrange_{\Y, \widehat V_-}^{(m)}\cap \xi[0, (1-\varepsilon^r)n ]^{2\varepsilon n^{1/2}} \not=\emptyset,  \brwrange_{\Y, \widehat V_-}^{((1+\lambda)m)} \cap \xi[0, n]=\emptyset \Big) . 
$$

Let \begin{align*}
\tau_{m,n}:=\inf\{i\ge 0: \dd(\xi_i, \brwrange_{\Y, \widehat V_-}^{(m)})\le  2\varepsilon n^{1/2}\}.
\end{align*} 

\noindent We have \begin{eqnarray}    p'''_{m, n}  &=&
\P_0\otimes \rwp_{x_n} \Big( \tau_{m, n} \le (1-\varepsilon^r)n ,  \brwrange_{\Y, \widehat V_-}^{((1+\lambda)m)} \cap \xi[0, n]=\emptyset \Big)
\nonumber
\\
&\le&
\E_0\Big(\max_{\dd(y, \brwrange_{\Y, \widehat V_-}^{(m)})\le  2\varepsilon n^{1/2}, \, y\in \z^5}  \rwp_y(\xi[0, \varepsilon^r n ]\cap \brwrange_{\Y, \widehat V_-}^{((1+\lambda)m)}=\emptyset) \Big) ,
\label{p'''mn} \end{eqnarray}

\noindent where the inequality follows from the strong Markov property of $\xi$ at $\tau_{m, n}$.  Let $$A_m:= \Big\{\#(\httree -[0, (1+\lambda) m ]\cap \X) < \lambda m/2\Big\}.$$

\noindent On $A_m$, we have \begin{equation}    \brwrange_{\Y, \widehat V_-}^{(m)} \subset V_\Y[0, m] \subset V_\Y[0, (1+\frac\lambda2)m] \subset \brwrange_{\Y, \widehat V_-}^{((1+\lambda)m)}. \label{Am2} \end{equation}

\noindent
By Lemma \ref{lem:easy_Z}, there is some positive constant $c_\delta$ such that   $$\max_{m \in [ \delta n^2, \frac1{\delta} n^2]} \P_0\Big(A_m^c \Big) \le   e^{-c_\delta n^2},$$

\noindent which, in view of \eqref{p'''mn} and \eqref{Am2}, implies that 
\begin{eqnarray*}    
p'''_{m, n} 
&\le & 
 e^{-c_\delta n^2}+
\E_0\Big(   1_{A_m}\max_{\dd(y, \brwrange_{\Y, \widehat V_-}^{(m)})\le  2\varepsilon n^{1/2}, \, y\in \z^5}  \rwp_y(\xi[0, \varepsilon^r n ]\cap \brwrange_{\Y, \widehat V_-}^{((1+\lambda)m)}=\emptyset) \Big)  
\\
&\le&
  e^{-c_\delta n^2}+
\E_0\Big(   \max_{\dd(y, V_\Y[0,m]) \le  2\varepsilon n^{1/2}, \, y\in \z^5}  \rwp_y(\xi[0, \varepsilon^r n ]\cap V_\Y[0, (1+\frac\lambda2)m]=\emptyset) \Big) .\end{eqnarray*}

Recall that $\zeta \in (0, \frac14-\frac1{q})$.  Let $$B_\Y:= \Big\{\max_{0\le k \le  \varepsilon^{-1/\zeta}} \max_{0\le j \le \varepsilon^{1/\zeta} m} | V_\Y(j+ k  
\varepsilon^{1/\zeta} m)- V_\Y(k \varepsilon^{1/\zeta} m )| < \varepsilon  n^{1/2}\Big\}.$$

On the event $ B_\Y$,   for any $y\in \z^d$ such that $\dd(y, V_\Y[0,m])
\le  2\varepsilon n^{1/2}$, there is some $0\le k <
\varepsilon^{-{1/\zeta}}$ such that $|y-V_\Y(k\varepsilon^{1/\zeta}
m)|\le 3 \varepsilon n^{1/2}< 4 \varepsilon n^{1/2}$, furthermore
\[
  \rwp_y(\xi[0, \varepsilon^r n ]\cap V_\Y[0,
  (1+\frac\lambda2)m]=\emptyset)  \le \rwp_y(\xi[0, \varepsilon^r n]\cap
  V_\Y[k\varepsilon^{1/\zeta} m,k\varepsilon^{1/\zeta} m+\lambda
  m/2]=\emptyset).
\]
It follows that\begin{equation}    \E_0\Big(   \max_{\dd(y, V_\Y[0,m]) \le  2\varepsilon n^{1/2}, \, y\in \z^5}  \rwp_y(\xi[0, \varepsilon^r n ]\cap V_\Y[0,(1+\frac\lambda2)m]=\emptyset) \Big)
\le
\P_0(B_\Y^c)+  J_{\eqref{cut-1}},  \label{cut-1} \end{equation}

\noindent with  $$J_{\eqref{cut-1}} :=  \E_0\Big(\max_{0\le k <  \varepsilon^{-1/\zeta}} \max_{|y-V_\Y(k\varepsilon^{1/\zeta} m)|\le 4 \varepsilon n^{1/2}, y\in \z^5}  \rwp_y(\xi[0, \varepsilon^r n]\cap V_\Y[k\varepsilon^{1/\zeta} m,k\varepsilon^{1/\zeta} m+\lambda m/2]=\emptyset)\Big). $$

 Applying \eqref{eq:max_diam2-Y++}  to $ \eta=\varepsilon^{1/\zeta}, r= \varepsilon \delta^{1/4}$ and $p$ sufficiently large such that $p (\frac1{4\zeta}- 1)- \frac1\zeta \ge 1$ and $p\ge  \frac{4}{\frac14-\frac1\zeta}$,  we get that uniformly in $m \in [ \delta n^2, \frac1{\delta} n^2]$, $$ \P_0(B_\Y^c) \le  C_\delta  \, \varepsilon + o_n(1),
 $$
 
 \noindent where $C_\delta:= \frac{C}{\mu(0)} \delta^{-p/4}$ in notation of \eqref{eq:max_diam2-Y++}. Therefore we have shown that uniformly in $m \in [ \delta n^2, \frac1{\delta} n^2]$, \begin{equation} p'''_{m, n} 
 \le  
 C_\delta \, \varepsilon + o_n(1)  +     J_{\eqref{cut-1}}. \label{p'''-08} \end{equation}

 For $J_{\eqref{cut-1}}$,  we deduce from \eqref{eq:t0_trick} that
\begin{equation}   
J_{\eqref{cut-1}}
\le    
 \frac 1 {\mu(0)}\E_0\Big( \max_{0\le k <  \varepsilon^{-{1/\zeta}}} U_k\Big), \label{maxUk}
 \end{equation}
where for $\varepsilon$ small enough so that $\varepsilon n^2\leq
\lambda m/2$ and for each $k$,
\begin{align*}  
U_k:=\max_{|y-\tbrw+{}(k\varepsilon^{1/\zeta} m)|\le 4 \varepsilon
  n^{1/2}, \, y\in \z^5}  \rwp_y(\xi[0, \varepsilon^r
  n]\cap\tbrw+{}[k\varepsilon^{1/\zeta} m,k\varepsilon^{1/\zeta}
  m+\varepsilon n^2]=\emptyset)  .  
\end{align*}

To estimate $\max_{0\le k <  \varepsilon^{-{1/\zeta}}} U_k$, we cut $V_+$ into smaller pieces. Consider the event $$B_{V_+}= B_{V_+}(m, n):= \Big\{\max_{0 \le i < j \le (1+\frac\lambda2)m, \, j-i \le \varepsilon    n^2}  |V_+(j)- V_+(i)| \le \varepsilon^{  \zeta} n^{1/2}\Big\}.$$

\noindent By \eqref{increment-V+},  we deduce from the union bound and  the translation invariance for $V_+$ in \eqref{eq:translational_invariance} that for some positive constant $a'$, \begin{equation}\max_{m \in [ \delta n^2, \frac1{\delta} n^2]}\P_0\big(B_{V_+}^c\big) \le C_\delta\, \varepsilon^{ a'}. \end{equation}

\noindent It follows that \begin{equation} \E_0\Big( \max_{0\le k <  \varepsilon^{-{1/\zeta}}} U_k\Big)
 \le C_\delta\, \varepsilon^{  a'} + \E_0\Big( \max_{0\le k <  \varepsilon^{-{1/\zeta}}}   U_k 1_{B_{V_+}}\Big)
\le
C_\delta\, \varepsilon^{  a'} +   \E_0\Big( \max_{0\le k <  \varepsilon^{-{1/\zeta}}} q^{(k)}_{\varepsilon, r}(n)  \Big), \label{max-Uk2}
\end{equation}

\noindent where for each $k\ge 0$,
\begin{align*}  
q^{(k)}_{\varepsilon, r}(n):=U_k 1_{\{\max_{0\le i \le  \varepsilon n^2} |V_+(k \varepsilon^{1/\zeta} m +i)- V_+(k \varepsilon^{1/\zeta} m)| \le \varepsilon^{\zeta} n^{1/2} \}}.  
\end{align*}

By \eqref{eq:translational_invariance}, $(q^{(k)}_{\varepsilon, r}(n))_{0\le k< \varepsilon^{-1/\zeta}}$ are identically distributed as $q_{\varepsilon, r}(n)$, which was defined in \eqref{def:qrepilson}.  Applying Lemma \ref{l:xi-new} to an arbitrary  $M>\frac1\zeta$,  there exists some   $\upsilon$ such that for all large $n$, 
$$ \P_0(q_{\varepsilon, r}(n)> \varepsilon^\upsilon) \le C\,
\varepsilon^{M} +o_n(1).$$
By union bounds, 
\begin{equation*} 
\P_0\pars*{\max_{0\le k <  \varepsilon^{-{1/\zeta}}} q^{(k)}_{\varepsilon, r}(n)  > \varepsilon^\upsilon}
\le \varepsilon^{-{1/\zeta}}\, \P_0(q_{\varepsilon, r}(n)> \varepsilon^\upsilon) 
 \le C\,  \varepsilon^{M-{1/\zeta}}.   
\end{equation*}  
Since $q^{(k)}_{\varepsilon, r}(n)\le 1$, we have
\begin{equation*} 
\E_0\pars*{
\max_{0\le k <  \varepsilon^{-{1/\zeta}}}q^{(k)}_{\varepsilon, r}(n)} \le  \varepsilon^\upsilon+ C\,\varepsilon^{M-{1/\zeta}} +o_n(1)   .   
\end{equation*}

\noindent Going back to \eqref{max-Uk2}, we get that $$ \E_0\Big( \max_{0\le k <  \varepsilon^{-{1/\zeta}}} U_k\Big)
 \le C_\delta\, \varepsilon^{  a'} + \varepsilon^\upsilon+ C\,\varepsilon^{M-{1/\zeta}}+o_n(1) ,$$
 
 \noindent which  together with \eqref{def-p'''mn},  \eqref{p'''-08} and \eqref{maxUk}, imply that 
 for all large $n$ and $m \in [ \delta n^2, \frac1{\delta} n^2]$, 
\begin{equation*}  
p''_{m, n} \le    o_n(1)  + C_\delta\, \varepsilon + C_\delta\, \varepsilon^{  a'}+\varepsilon^\upsilon +  C\, \varepsilon^{M-{1/\zeta}},   
\end{equation*} 
implying   \eqref{estimate:p''} and completing the proof of Proposition \ref{l:p(m,n)}. 
\end{proof}

\subsection{Proof of Proposition \ref{p:intersection}} 
\label{sec:3.4}

\begin{figure}[ht]
\centering
\begin{minipage}{.5\linewidth}
  \centering
  \includegraphics[width=.9\linewidth]{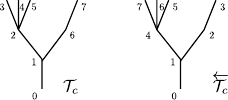}
  \caption{\scriptsize An illustration of $\ttree{c}$ and $\ottree{c}$.}
\label{fig:tree_reversed}
\end{minipage}%
\begin{minipage}{.5\linewidth}
  \centering
  \includegraphics[width=.6\linewidth]{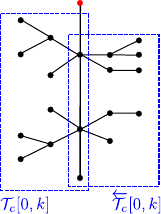}
\caption{\scriptsize
In order that the vertex on the top does not belong to $\ttree c[0,k]\cup\ottree c[0,k]$, there must be at least $2k+1-\#\ttree c$ points on the line connecting it to the root. In this figure, $k=10$.}
\label{fig:tree_overlap}
\end{minipage}
\end{figure}

As shown in Fig.~\ref{fig:tree_reversed}, 
we write $\ttree c(0), \ttree c(1),\dots$ as $\ttree c$ in its depth-first order, 
and we write  $\ottree c(0), \ottree c(1),\dots$ as the depth-first order in the reversed sense.

Observe that if a vertex does not belong to $\ttree c[0,\frac 3 5\#\ttree c]\cup\ottree c[0,\frac 3 5\#\ttree c]$, then the line connecting it to the root contains at least $\frac 1 5\#\ttree c$ points, as shown in Fig.~\ref{fig:tree_overlap}. Hence there is some positive constant $C$ such that for every $\delta>0,n\ge 1$,
\begin{equation}\label{eq:j0}
\begin{aligned}
&\mathbf P \pars*{\ttree c\not\subset \ottree c\bracks*{0,\frac 3 5\#\ttree c}\cup \ttree c\bracks*{0,\frac 3 5\#\ttree c}, \#\ttree c\in[\delta n^2,n^2/\delta]}\\
\le&\P\pars*{\text{height of $\ttree c$ is at least }\delta n^2/5}
\\
\le & C\,  \delta^{-1}n^{-2}
\end{aligned}
\end{equation}

\noindent where the last inequality follows from the classical estimate on the branching process: \begin{equation}    \P\pars*{\text{height of $\ttree c$ is at least } j} \sim \frac{2}{\sigma^2 j}, \qquad j\to \infty. \label{height-Tc} \end{equation}

For simplicity we may and will assume that the offspring distribution $\mu$ is aperiodic, as one can be easily adapt the proof line by line for the periodic case.
In particular, given aperiodicity, by Dwass \cite{dwass} and local central limit theorem (see \cite[Theorem 2.3.9]{Lawler-Limic}),
\begin{align}\label{eq:ctree} 
 n^{3/2}  \P(\#\ttree c=n) \sim \frac{1}{\sigma\sqrt{2\pi}} , \quad n \to\infty. 
\end{align}

\noindent Therefore \begin{equation} \P(\#\ttree c > j) \sim \frac{2}{\sigma \sqrt{2\pi j}}, \qquad j\to \infty. \label{total-Tc}   \end{equation}

\noindent 
 We are now ready to prove Proposition \ref{p:intersection}.
\begin{proof}[Proof of Proposition \ref{p:intersection}]    Let $\delta>0$ be small, we have
\begin{eqnarray*} 
&& I(\varepsilon, n) \le  \mathbf P \pars*{\ttree c\not\subset \ottree c\bracks*{0,\frac 3 5\#\ttree c}\cup \ttree c\bracks*{0,\frac 3 5\#\ttree c}, \#\ttree c\in[\delta n^2,n^2/\delta]}
\nonumber \\
\quad && +2\sum_{m= \delta n^2}^{n^2/\delta} \mathbf P_0\otimes \rwp_{x_n} 
\pars*{ \tbrw c{}\bracks*{0,\frac 3 5m} \cap (\xi[0,n])^{\varepsilon n^{1/2}} \not=\emptyset, \tbrw c{}\bracks*{0,{\frac 4 5}m} \cap \xi[0,n] =\emptyset, \#\ttree c=m} 
\nonumber \\
\quad && +\P(\#\ttree c> n^2/\delta) + \P_{x_n}  \big( \#\ttree c < \delta n^2 , \, \brwrange_c \cap  \ball(0,\eta |x_n|+\varepsilon n^{1/2})  \not=\emptyset\big)\nonumber
\\
&&=:  J_0(\delta,n)
+2\sum_{m= \delta n^2}^{n^2/\delta} J_1(\varepsilon, n,m) 
+   J_2(\delta, n)+ J_3(\delta, \varepsilon, n).\nonumber
\end{eqnarray*}

We have already bounded $J_0$ in \eqref{eq:j0}. 
For $J_2$, we deduce from   \eqref{total-Tc} that for some positive constant $C$, \begin{equation*}  J_2(\delta, n)= \P(\#\ttree c> n^2/\delta)  \le C\,  \delta^{1/2} n^{-1}.  \end{equation*}

 For $J_3$, when $\varepsilon<(1-\eta)|x|/4$, we have for all large $n$, $\brwrange_c \cap  \ball(0,\eta |x_n|+\varepsilon n^{1/2})  \not=\emptyset$ implies that there is some $z\in \brwrange_c$ such that $|z-x_n| \ge |x_n|- (\eta |x_n|+\varepsilon n^{1/2}) \ge \frac12 (1-\eta)|x| n^{1/2}.$ It follows that   
\begin{eqnarray*} 
J_3(\delta, \varepsilon, n)
& \le&
 \P_0 \pars*{ \max_{z\in \brwrange_c} |z| \ge \frac12 (1-\eta)|x| n^{1/2},  \#\ttree c < \delta n^2 }.
\end{eqnarray*}

\noindent By \eqref{max-Vc}, there exists some $C>0$ such that for all $m\ge 1$, $\E_0\parsof*{\max_{z\in \brwrange_ c}|z|^q}{\#\ttree c=m} \le C\,  m^{q/4} $. Using Markov's inequality and  \eqref{eq:ctree},  we have that for some positive constants $C_{\eta, x}$ and $C'_{\eta, x} $, 
\begin{align*}  
J_3(\delta, \varepsilon, n)
=&
\sum_{m=1}^{\delta n^2}\P_0(\#\ttree c=m)\P_0\parsof*{\max_{z\in \brwrange_c}|z|\ge \frac12 (1-\eta)|x| n^{1/2}}{\#\ttree c=m}\\
\le & C_{\eta, x} n^{-q/2} \, \sum_{m=1}^{\delta n^2} m^{-3/2}
 \E_0\big(\max_{z\in \brwrange_c}|z|^{q} \, |\, \#\ttree c=m\big)
 \\
\le  &  C'_{\eta, x}  \, \delta^{1/2}\, n^{-1}.  
\end{align*} 
 
Finally, by taking $k=4m/5$ in \eqref{eq:absolutecontinuity2}, we can find some universal constant $C>0$ such that $J_1(\varepsilon, n, m)$ is bounded above by
\begin{align*}
C\, \P_0\otimes \rwp_{x_n} \pars*{\htbrw-{}[0, {3m}/{5}] \cap (\xi[0,n])^{\varepsilon n^{1/2}} \not=\emptyset,  \htbrw-{}[0, {4m}/{5}] \cap \xi[0,n] =\emptyset }\P(\#\ttree c=m).
\end{align*}
Then by Proposition \ref{l:p(m,n)}, we deduce that for some positive constants $c$,   $C_\delta$ and $C'_\delta$ such that  for all $\delta n^2\le m \le n^2/\delta$, 
\[
J_1(\varepsilon, n, m)
\le (C_\delta\,   \varepsilon^c +o_n(1))\, \P(\#\ttree c=m)  ,
\]
hence by  \eqref{eq:ctree}
\[
\sum_{m= \delta n^2}^{n^2/\delta} J_1(\varepsilon, n, m)  
\le (C_\delta\,   \varepsilon^c+ o_n(1))\, \P(\#\ttree c \ge \delta n^2)
\le   C'_\delta\, ( \varepsilon^c\,  + o_n(1))\, n^{-1} .
\]

Combine the above estimates, we see that  \begin{align*}
\limsup_{n\to \infty} n\, \Big(J_0(\delta,n)
+2\sum_{m= \delta n^2}^{n^2/\delta} J_1(\varepsilon, n,m) 
+   J_2(\delta, n)+ J_3(\delta, \varepsilon, n)\Big) 
\le C\,  \delta^{1/2}+C'_{\eta, x}  \, \delta^{1/2} + C'_\delta\,  \varepsilon^c,
\end{align*}
and we conclude Proposition \ref{p:intersection} by letting $\varepsilon\to0$ then $\delta\to 0$.\end{proof}

We end the paper by a remark on the condition $\{\xi[0,n]\subset\ball\pars*{0,\eta{|x_n|}}\}$ in Proposition \ref{p:intersection}:

\begin{remark}  \label{r:B-eta}  
Without the condition $\xi[0,n]\subset\ball\pars*{0,\eta{|x_n|}}$, Proposition \ref{p:intersection} is no longer true. In fact, we have $\brwrange_c=\{x_n\}$ with probability $\mu(0)>0$, then 
\begin{align*}
&\,n\,\mathbf P_{x_n}\otimes \rwp_0 \pars*{ \brwrange_c \cap \cnbd{(\xi[0,n])}{{\varepsilon \sqrt n}} \not=\emptyset, \brwrange_c \cap \xi[0,n] =\emptyset
}\\
&\ge   \, \mu(0)  \,n\,\rwp_0 \pars*{ \{x_n\} \cap \cnbd{(\xi[0,n])}{{\varepsilon \sqrt n}} \not=\emptyset, \{x_n\} \cap \xi[0,n] =\emptyset}\\
&\ge   \,\mu(0)  \, n\,\big(\rwp_0 \pars*{ \ball(x\sqrt n,\varepsilon\sqrt n) \cap \xi[0,n] \not=\emptyset}-\rwp_0 \pars*{\{x_n\} \cap \xi[0,n] \neq \emptyset}\big)
\\
& \ge   \mu(0)  \, c_\varepsilon\, n 
\end{align*}
which diverges as $n\to\infty$, where the last inequality follows from the facts that $c_\varepsilon:=\liminf_{n\to\infty} \rwp_0 \pars*{ \ball(x\sqrt n,\varepsilon\sqrt n) \cap \xi[0,n] \not=\emptyset} >0$ and $\lim_{n\to\infty}\rwp_0 \pars*{\{x_n\} \cap \xi[0,n] \neq \emptyset}= 0$.  
\end{remark}


\begin{thebibliography}{99}
\baselineskip=10pt

\bibitem{AS20} Asselah, A. and  Schapira, B. (2020). Deviations for the capacity of the range of a random walk. {\it 
Electron. J. Probab.} {\bf 25}  1--28.

\bibitem{AS20+} Asselah, A. and  Schapira, B. (2023). Extracting subsets maximizing capacity and Folding of Random Walks. {\it Ann. Sci. \'{E}c. Norm. Sup\'{e}r. (4)} {\bf 56} 1565--1582.


 \bibitem{asselah2023local}  Asselah, A.,    Schapira, B. and   Sousi, P. (2023+).  Local times and capacity for transient branching random walks. {\it arXiv preprint arXiv:2203.03188}


 \bibitem{ASS18}   Asselah, A.,    Schapira, B. and   Sousi, P. (2018). Capacity of the range of random walk on $\z^d$. {\it Trans. Am. Math. Soc., } {\bf 370} 7627--7645.


 \bibitem{ASS19}   Asselah, A.,    Schapira, B. and   Sousi, P. (2019). Capacity of the range of random walk on $\z^4$. {\it   Ann. Probab.} {\bf 47} 1447--1497.


\bibitem{AOSS}
A.~Asselah, I.~Okada, B.~Schapira, and P.~Sousi.
\newblock Branching random walks and Minkowski sum of random walks.
\newblock {\em arXiv preprint arXiv:2308.12948}, 2023.


\bibitem{bdh-23+} Bai, T., Delmas, J.F. and Hu, Y. (2024+). Branching capacity and Brownian snake capacity. 
\newblock {\em arXiv preprint arXiv:2402.13735} 

\bibitem{bai2022convergence} Bai, T. and Hu, Y. (2023). Convergence in law for the capacity of the range of a critical branching random walk. {\it Ann. Appl. Probab.} {\bf 33}   4964--4994. 

 \bibitem{bai2020capacity}
  Bai, T. and  Wan, Y.  (2022). Capacity of the range of tree-indexed random walk. {\it Ann. Appl. Probab.} {\bf 32}   1557--1589. 


\bibitem{Barlow-Yor}
	Barlow, M.T. and  Yor, M. (1982).
	Semimartingale inequalities via the {G}arsia-{R}odemich-{R}umsey
	lemma, and applications to local times.
	{\it J. Functional Analysis}, 49(2):198--229.


\bibitem{BC-12} Benjamini, I. and Curien, N. (2012). Recurrence of the $\z^d$-valued infinite snake via unimodularity. {\it Electron. Commun. Probab.} {\bf 17} 1--10.


\bibitem{chang17} Chang, Y. (2017).  Two observations on the capacity of the range of simple random walks on $\z^3$ and $\z^4$. {\it Electron. Commun. Probab. } {\bf 22} 1--9.

\bibitem{DO22+} Dembo, A. and Okada, I. (2022+). Capacity of the range of random walk: The law of the iterated logarithm. \newblock {\em arXiv preprint arXiv:2208.02184}



\bibitem{DE51} Dvoretzky, A. and  Erd\H{o}s, P. (1951). Some problems on random walk in space. {\it Proceedings of the second Berkeley symposium on mathematical statistics and probability.}

\bibitem{dwass}
 Dwass, M. (1969). 
\newblock The total progeny in a branching process and a related random walk.
\newblock {\em J. Appl. Probability}, {\bf 6} 682--686.


\bibitem{HS23} Hutchcroft, T. and Sousi, P. (2023). 
Logarithmic corrections to scaling in the four-dimensional uniform spanning tree. {\it 
Commun. Math. Phys.} {\bf 401} 2115--2191. 


\bibitem{JP71} Jain, N.C. and Pruitt, W.E. (1971).
 The range of transient random walk.   {\it 
J. Anal. Math.} {\bf 24} 369--393. 

\bibitem{JO68}  Jain, N.C. and Orey, S. (1968). 
On the range of random walk.  {\it 
 Israel Journal of Mathematics} {\bf 6}  373–-380. 



\bibitem{JO73}  Jain, N.C. and Orey, S. (1973). 
Some properties of random walk paths.  {\it 
J. Math. Anal. Appl. } {\bf 43}  795--815. 




\bibitem{Lawler-Limic}   Lawler, G.F. and   Limic, V. (2010). {\it Random walk: A modern introduction.} Cambridge University Press.

\bibitem{LG86a} Le Gall, J.F. (1986). 
Propriétés d’intersection des marches aléatoires. I: Convergence vers le temps local d’intersection. {\it Commun. Math. Phys.} {\bf 104} 471--507.


\bibitem{LG86b} Le Gall, J.F. (1986). 
Propriétés d’intersection des marches aléatoires. II: Étude des cas critiques.  {\it Commun. Math. Phys.} {\bf 104} 509--528.


\bibitem{LeGall1999} Le Gall, J.F. (1999). {\it Spatial Branching Processes, Random Snakes and Partial Differential Equations.}  Birkhäuser, Basel. 

 
 
 
 \bibitem{LeGall-Lin-range} Le Gall, J.F. and   Lin, S. (2016).  The range of tree-indexed random walk. {\it J. Inst. Math. Jussieu}  {\bf 15} 271--317.
 

\bibitem{PPS} Pemantle, R., Peres, Y. and Schapiro, J.W. (1996). The trace of spatial Brownian motion is capacity-equivalent to the unit square. {\it Proba. Th. Rel. Fields.} {\bf 106} 379--399.


 \bibitem{Petrov} Petrov, V.V. (1995). {\it Limit Theorems of Probability Theory. Sequences of independent random variables.} Clarendon Press, Oxford. 




\bibitem{schapira20} Schapira, B. (2020). 
Capacity of the range in dimension $5$. {\it 
Ann. Probab.} {\bf 48}   2988--3040. 

\bibitem{schapira2023branching} Schapira, B. (2023+). Branching capacity of a random walk range. {\it  arXiv preprint arXiv:2303.17830}


\bibitem{MR1625467}
 Uchiyama, U. (1998).
\newblock Green's functions for random walks on {${\bf Z}^N$}.
\newblock {\em Proc. London Math. Soc. (3)}, {\bf 77} 215--240.

\bibitem{zhu2016critical} Zhu, Q. (2016+).  On the critical branching random walk I: Branching capacity and
  visiting probability. {\it  preprint arXiv:1611.10324.} 


\bibitem{zhu-2} Zhu, Q.  (2016+).  On the critical branching random walk II: Branching capacity and
  branching recurrence. {\it  preprint arXiv:1612.00161}


\bibitem{zhu2021critical}  Zhu, Q. (2021).  On the critical branching random walk III: The critical dimension. {\it  Ann. Inst. H. Poincar\'e Probab. Statist.} {\bf 57} 73--93. 

\end{thebibliography}
\end{document}